\documentclass[11pt,a4paper]{amsart}
\usepackage[utf8]{inputenc}
\usepackage[T1]{fontenc}
\usepackage{textcomp}
\usepackage{a4wide}
\usepackage{mathrsfs}
\usepackage{url}
\usepackage{lmodern}
\usepackage{setspace}
\newcommand{\indi}{\mathbf{1}}
\usepackage{enumitem}
\usepackage{amsmath,amssymb}
\newtheorem{theorem}{Theorem}[section]
\newtheorem{lemma}[theorem]{Lemma}
\newtheorem{prop}[theorem]{Proposition}
\newtheorem{corollary}[theorem]{Corollary}

\theoremstyle{remark}

\theoremstyle{definition}
\newtheorem{definition}[theorem]{Definition}
\newtheorem{remark}[theorem]{Remark}
\newtheorem{example}[theorem]{Example}

\newtheorem*{lemma*}{Lemma}
\newtheorem*{theorem*}{Theorem}
\newtheorem*{proposition*}{Proposition}
\newtheorem*{corollary*}{Corollary}
\newtheorem*{definition*}{Definition}

\usepackage{tikz}
\usetikzlibrary{arrows,shapes,trees}

\newcommand{\C}{\mathbb{C}}
\newcommand{\N}{\mathbb{N}}
\newcommand{\Q}{\mathbb{Q}}
\newcommand{\R}{\mathbb{R}}

\newcommand{\Z}{\mathbb{Z}}

\newcommand{\brac}[2]{\langle#1,#2\rangle}

\newcommand{\entr}[1]{\mathbb{E}(#1)}

\usepackage{tikz}
\usetikzlibrary{arrows,shapes,trees}

\newcommand{\wtilde}[1]{\widetilde{#1}}
\newcommand{\Frac}[2]{\frac{#1}{#2}}

\newcommand{\what}[1]{\hat{#1}}

\usepackage{mathtools}

\makeatletter
\DeclareRobustCommand\widecheck[1]{{\mathpalette\@widecheck{#1}}}
\def\@widecheck#1#2{%
    \setbox\z@\hbox{\m@th$#1#2$}%
    \setbox\tw@\hbox{\m@th$#1%
       \widehat{%
          \vrule\@width\z@\@height\ht\z@
          \vrule\@height\z@\@width\wd\z@}$}%
    \dp\tw@-\ht\z@
    \@tempdima\ht\z@ \advance\@tempdima2\ht\tw@ \divide\@tempdima\thr@@
    \setbox\tw@\hbox{%
       \raise\@tempdima\hbox{\scalebox{1}[-1]{\lower\@tempdima\box
\tw@}}}%
    {\ooalign{\box\tw@ \cr \box\z@}}}
\makeatother
\title[Weighted inequalities and uncertainty principles]{Weighted inequalities and uncertainty principles for the $\boldsymbol{(k,a)}$-generalized Fourier transform}
\author{Troels Roussau Johansen}
\address{Mathematisches Seminar, Chr.-Albrechts--Universit\"at zu Kiel, \\ Ludewig-Meyn--Strasse 4, 24118 Kiel, Germany}
\email{johansen@math.uni-kiel.de}
\subjclass[2010]{Primary: 33C52, Secondary: 26D10, 43A15, 43A32, 44A15, 46B70}
\keywords{Hausdorff--Young, Hardy--Littlewood, deformed Dunkl transform, Hankel transform, uncertainty principles}
\begin{document}
\begin{abstract}
We obtain several versions of the Hausdorff--Young and Hardy--Littlewood inequalities for the $(k,a)$-generalized Fourier transform recently investigated at length by Ben Sa\"\i d, Kobayashi, and \O rsted. We also obtain a number of weighted inequalities -- in particular Pitt's inequality --  that have application to uncertainty principles. Specifically we obtain several analogs of the Heisenberg--Pauli--Weyl principle for $L^p$-functions, local Cowling--Price-type inequalities, Donoho--Stark-type inequalities and qualitative extensions. We finally use the Hausdorff--Young inequality as a means to obtain entropic uncertainty inequalities.
\end{abstract}

\maketitle

\begin{small}
\tableofcontents
\end{small}
\section{Introduction}
Uncertainty principles have long been a mainstay of mathematical physics and classical Fourier analysis alike and are statements of the form that a function and its Fourier transform cannot both be small. A well-known example is the Heisenberg--Pauli--Weyl uncertainty principle to the effect that position and momentum of a quantum particle cannot both be sharply localized. In terms of Fourier analysis it can be paraphrased as the statement that if $f\in L^2(\R^n)$ and $\alpha>0$,
\begin{equation}\label{HPW1}
\|f\|_2^4\leq c_\alpha\Bigl(\int_{\R^n}\vert x\vert^{2\alpha}\vert f(x)\vert^2\,dx\Bigr)\Bigl(\int_{\R^n}\vert\xi\vert^{2\alpha}\vert\widehat{f}(\xi)\vert^2\,d\xi\Bigr),
\end{equation}
or in terms of the Laplace operator $\Delta$ as
\begin{equation}\label{HPW2}
\|f\|_2^4\leq c_\alpha\Bigl(\int_{\R^n}\vert x\vert^{2\alpha}\vert f(x)\vert^2\,dx\Bigr) \Bigl(\int_{\R^n}\vert(-\Delta)^{\alpha/2}f(x)\vert^2\,dx\Bigr).
\end{equation} 

Many variations and extensions are outlined in the excellent survey \cite{Folland-Sitaram}, as well as \cite{Cowling-Price}, where the following qualitative version is also explained. Consider the sets
\[A_f=\{x\in\R^n\,:\, f(x)\neq 0\}\quad\text{and}\quad A_{\what{f}}=\{\xi\in\R^n\,:\,\what{f}(\xi)\neq 0\}\]
or more generally the analogous sets for functions on a locally compact abelian group. It is easy to prove that if $f\in L^2(\R^n)\setminus\{0\}$, then $|A_f|\cdot|A_{\what{f}}|\geq 1$. This is originally due to Matolcsi and Sz\"ucs \cite{Matolcsi-Szucs} and was strengthened considerably by Benedicks, cf. \cite{Benedicks}:
\begin{theorem} If $f\in L^1(\R^n)$ and $|A_f|\cdot|A_{\what{f}}|<\infty$, then $f=0$ almost everywhere.
\end{theorem}
A different proof based on operator theory was given in \cite{Amrein-Berthier} and also yield complementary results that we shall discuss in a later section. We recently established analogues of the Matolcsi--Sz\"ugs and more generally the Benedicks--Amrein-Berthier theorems in the framework of harmonic analysis in root systems, and we shall presently establish their analogues for the $(k,a)$-generalized transform $\mathcal{F}_{k,a}$ that will be described later in this introduction.

A variation of such qualitative statements is obtained by allowing $f$ and $\what{f}$ to be negligible small on the complements of given sets $A$, $B$. To fix notation, let $G$ be a locally compact abelian group with dual group $\widehat{G}$, and let $A\subset G$, $B\subset\widehat{G}$ be measurable subsets. Consider the orthogonal projections $P_A$, $Q_B$ on $L^2(G)$ defined by $P_Af=\indi_A f$ and $\widehat{Q_Bf}=\indi_B\what{f}$ respectively. The operator $P_AQ_B$ -- which also intervenes in \cite{Amrein-Berthier} -- is a Hilbert--Schmidt operator, and the essence of the Donoho--Stark uncertainty principle is a statement of the following form: If there is a nonzero $f\in L^2(G)$ such that $\|\indi_{G\setminus A}f\|_2\leq\epsilon\|f\|_2$ and $\|\indi_{\widehat{G}\setminus B}\what{f}\|_2\leq\delta \|\what{f}\|_2$ for given constants $\epsilon,\delta>0$, then $1-\epsilon-\delta\leq\|P_AQ_B\|_{2\to 2}$.

The third version of an uncertainty principle is related to the Heisenberg--Pauli--Weyl inequality but is formulated in terms of the Shannon entropy instead and therefore stronger, cf. \cite[Section~5]{Folland-Sitaram}. Following Shannon, the \emph{entropy} of a probability density function $\rho$ on $\R^n$ is defined by
\[\entr{\rho}=-\int_{\R^n}\rho(x)\log(\rho(x))\,dx.\]
Hirschman defined entropy without the negative sign but we have adopted the definition from \cite{Folland-Sitaram}. Given a function $f\in L^2(\R)$ such that $\|f\|_2=1$, it was observed by Hirschman \cite{Hirschman-entropy} that $\entr{\vert f\vert^2}+\entr{\vert\widehat{f}\vert^2}\geq 0$, and he made the conjecture that 
\begin{equation}\label{eqn.Hirschman-1D}
\entr{\vert f\vert^2}+\entr{\vert\widehat{f}\vert^2}\geq 1-\log 2>0.
\end{equation}
The proof by Hirschman was based on an endpoint differentiation technique applied to the Hausdorff--Young inequality
\begin{equation}\label{eq.Hausdorff-Young}
\|\widehat{f}\|_{p'}\leq c_p\|f\|_p,\qquad 1\leq p\leq 2,\quad \frac{1}{p}+\frac{1}{p'}=1,\end{equation}
and his argument carries over to the case of $\R^n$ without change. The analogoue of \eqref{eqn.Hirschman-1D} thereby becomes
\begin{equation}\label{eqn.Hirschman}
\entr{\vert f\vert^2}+\entr{\vert\widehat{f}\vert^2}\geq n(1-\log 2),\end{equation}
where $f\in L^2(\R^n)$ with $\|f\|_2=1$. Hirschman apparently raised the conjecture \eqref{eqn.Hirschman-1D} after having experimented with Gaussian functions in place of $f$, and indeed the conjectures \eqref{eqn.Hirschman-1D} and \eqref{eqn.Hirschman} are correct, as observed by Beckner in Section~IV.3 of \cite{Beckner-inequalities}. Among other things, Beckner's paper records the the optimal constant $c_p$  in \eqref{eq.Hausdorff-Young} for $p\in[1,2]$, thereby extending a result by Babenko \cite{Babenko}. It was furthermore proved that Gaussians are optimizers for the Hausdorff--Young inequality, so the method by Hirschman -- now applied to the sharp Hausdorff--Young inequality -- immediately establishes \eqref{eqn.Hirschman}. The same conclusion was made in \cite{BBM}. It was recently established (cf. Theorem~1.5 in  \cite{OP}) that normalized Gaussian do in fact serve as minimizers in \eqref{eqn.Hirschman-1D}, \eqref{eqn.Hirschman} (Hirschman anticipated such a result but due to the endpoint differentiation one cannot deduce this fact from a similar statement about the sharp Hausdorff--Young inequality).

We have recently investigated these topics in the case of the Cherednik--Opdam and Heckman--Opdam transforms associated with a root system, cf. \cite{Johansen-HY}, \cite{Johansen-unc}. The results in the present paper are complimentary, in the sense that while we also work in a framework of generalized harmonic analysis in root systems, the motivation and the resulting transform are different. In order to motivate the construction of the $(k,a)$-generalized Fourier transform $\mathcal{F}_{k,a}$ in \cite{BSKO}  we shall briefly recall several alternative descriptions of the the Euclidean Fourier transform $\mathcal{F}$, which is defined by

\[\mathcal{F}f(\xi)=\frac{1}{(2\pi)^{N/2}}\int_{\R^N} f(x)e^{-i\langle x,y\rangle} dx,\quad f\in L^1(\R^N).\]
Alternatively
\begin{equation}\label{eqn.F-kernel}
\mathcal{F}f(\xi)=\frac{1}{(2\pi)^{N/2}}\int_{\R^N} f(x)K(x,\xi)\,dx,
\end{equation}
where $K(x,\xi)$ is the unique solution to the system of partial differential equations $\partial_{x_j}K(x,\xi)=-i\xi_jK(x,\xi)$, $j=1,\ldots,N$ subject to the initial value condition $K(0,\xi)=1$ for $\xi\in\R^N$. A third description was discovered by R. Howe \cite{Howe88},
\begin{equation}\label{eqn.F-spectral}
\mathcal{F} = \exp\Bigl(\frac{i\pi N}{4}\Bigr) \exp\Bigl(\frac{i\pi}{4}(\Delta-\|x\|^2)\Bigr),
\end{equation}
where $\Delta$ is the Laplace operator on $\R^N$. 

Both of the representations \eqref{eqn.F-kernel} and \eqref{eqn.F-spectral} have their uses, and it is explained in the overview paper \cite{deBie-overview} how to construct various extensions such as a fractional Fourier transform and Clifford algebra-valued analogues. We are concerned with a different kind of extension, where the Euclidean Laplace operator $\Delta$ is replaced by the sum of squares $\Delta_k$ of Dunkl operators associated with a given finite reflection group in $\R^N$. The same $\mathfrak{sl}_2$-commutator relations continue to hold, and an analogue of \eqref{eqn.F-kernel} holds as well. It was observed in \cite{BSKO} that one can introduce an additional parameter to the Dunkl-operator construction, in terms of which the Euclidean harmonic oscillator is naturally replaced by an $a$-deformed Dunkl-harmonic oscillator $\|x\|^{2-a}\Delta_k-\|x\|^a$. The resulting spectrally defined family of operators $\mathcal{F}_{k,a}(z)=\exp(\frac{z}{a}(\|x\|^{2-a}\Delta_k-\|x\|^a))$, $\Re z\geq 0$,  may therefore be regarded as a two-parameter generalization of Howe's description \eqref{eqn.F-spectral}, where $k$ refers to a multiplicity function and $a>0$. The special case $a=2$ recovers the Dunkl transform in $\R^N$ and it is therefore natural to ask for analytical properties of $\mathcal{F}_{k,a}(i\frac{\pi}{2})$ such as a Plancherel theorem or an inversion formula.  The case $a=1$ is related to an integral transform appearing in work by Kobayashi and Mano (\cite{Kobayashi-Mano-appl}, \cite{Kobayashi-Mano-proc}, \cite{Kobayashi-Mano}) on the other hand. In particular, we obtain uncertainty principles for their integral transform at no additional cost.

These questions were addressed at length in \cite{BSKO} and placed in a wider context in \cite{deBie-overview}, \cite{deBie-et.al.}, \cite{deBie-trans}, and \cite{deBie15} but many additional questions were left often. Indeed our motivation was to extend classical results beyond the Plancherel theorem for $\mathcal{F}$ -- such as the Hausdorff--Young, Hardy--Littlewood and Pitt's inequalities -- to $\mathcal{F}_{k,a}$ and simultaneously investigating applications to uncertainty principles, given that the connection to quantum mechanics is apparent. Having a Hausdorff--Young inequality for $\mathcal{F}_{k,a}$ is of course key to the proof of the entropic inequality. Since the Hausdorff--Young inequality is easily established by means of interpolation, we found it natural to explore further weighted inequalities arising from more intricate interpolation arguments. These include Hardy--Littlewood inequalities of several kinds but the scope of interpolation is wider.

We have adopted the modern point of view of \cite{Benedetto-Heinig} that weighted inequalities such as Pitt's inequality should be obtained by interpolation arguments that do not rely on explicit information on the transform under consideration. We find this approach sensible, since one of the major technical obstacles in the further investigation of $\mathcal{F}_{k,a}$ is a lack of explicit formulae for the kernel that appears in the analogue of \eqref{eqn.F-kernel}. As already mentioned several classical uncertainty principles were recently \cite{Ghobber-Jaming-studia} established for a general class of integral transforms that includes the Dunkl transform, and we presently extend these principles and add further to the list of results. The guiding principle has therefore been to use the description of $\mathcal{F}_{k,a}$ as an integral transform in combination with interpolation arguments and spectral considerations. The main results may briefly be summarized as follows.

\begin{itemize}

\item We obtain an analogue of Hirschman's entropic inequality and use it to give a new proof of the Heisenberg--Pauli--Weyl  uncertainty principle recently obtained by Ben Sa\"\i d, Kobayashi, and \O rsted, cf. \cite[Theorem~5.29]{BSKO}, although not with a sharp constant.

\item We obtain large classes of weighted inequalities for $\mathcal{F}_{k,a}$, the most important one being  Pitt's inequality. These inequalities are based on rearrangement and interpolation techniques from \cite{Benedetto-Heinig} so the constants are not optimal.
We also establish several Hardy--Littlewood inequalities that are new already for the Dunkl transform $\mathcal{F}_{k,2}$.

\item We obtain a variation of the Heisenberg inequality involving a combination of $L^1$- and $L^2$-norms; the result was recently obtained for Dunkl transform by \cite{Ghobber-L1} and involve additional classical inequalities of Nash- and Clark-type.

\item The Heisenberg--Pauli--Weyl, Donoho--Stark, and Benedicks--Amrein-Berthier principles do not rely on having sharp constants and are established in general along the lines of \cite{Ghobber-Jaming-studia}. These results are collected towards the end of the paper as they do not require new proofs. We do provide a proof of the weaker Matolcsi--Sz\"ucs principle, though.
\end{itemize}

The point about \cite{Ghobber-Jaming-studia} is to use the representation of $\mathcal{F}_{k,a}$ as an integral operator with a well-behaved kernel and apply known uncertainty principles for such operators.  The point of departure, it seems, was the observation by de Jeu (cf. \cite{deJeu-uncertainty}) that a Donoho--Stark-type inequality established in the framework of Gelfand pairs by J. Wolf in \cite{Wolf-Nova}, \cite{Wolf-cayley} could be generalized to a large class of integral operators satisfying suitable  Plancherel-type estimates. A few years ago Ghobber and Jaming revisited the approach by de Jeu in the setting of the Hankel transform and recently (\cite{Ghobber-Jaming-studia}) extended the scope of their results even further to include, among others, the standard Dunkl transform on $\R^N$ (specifically we refer the reader to Theorem 4.3 and Theorem 4.4 in \cite{Ghobber-Jaming-studia}). 

While the connection between, say, the entropic inequality and the Heisenberg--Pauli--Weyl principle is well known in Euclidean analysis and nicely laid out in \cite{Folland-Sitaram}, it seems that a similar connection has gone unnoticed in more general settings such as Dunkl theory . At the same time we want to raise awareness of the interesting open question regarding sharp inequalities and the immediate applications to mathematical physics.

It must be pointed out that our version of Pitt's inequality is  \emph{not} strong enough to establish even a weak form of Beckner's logarithmic uncertainty principle. 
\medskip

\textbf{Notes added in proof:} After this paper was written, there has been further progress in the study of Pitt's inequality. In \cite{GIT1}, the authors establish a sharp Pitt's inequality for Dunkl transform in $L^2(\R^N)$ and in \cite{GIT2}, they obtain the sharp Pitt's inequality for the transform $\mathcal{F}_{k,a}$ in $L^2(\R^N)$ that we also consider. They also obtain related logarithmic uncertainty inequalities. In particular, these results apply to radial functions in $L^2(\R^N)$ where $\mathcal{F}_{k,a}$ is a generalized Hankel transform. 

\section{The deformed Fourier transform and interpolation theorems}
The present section is a brief overview of definitions and results for the deformed Dunkl-type harmonic oscillator introduced by Ben Sa\"\i d, Kobayashi, and \O rsted in \cite{BSKO-cr} and investigated in detail in \cite{BSKO} that will be needed later on. A subsection has been devoted to a discussion of the important case of radial functions, where the harmonic analysis simplifies significantly. Since interpolation in Lorentz spaces is not usually encountered in literature regarding harmonic analysis in root systems, we have included some technical remarks towards the end of the section for easy reference. 
\smallskip

Let $\brac{\cdot}{\cdot}$ denote the standard Euclidean inner product on $\R^N$ and let $\|\cdot\|$ be the associated norm. The reflection associated with a non-zero vector $\alpha\in\R^N$ is defined by $r_\alpha(x)=x-2\frac{\brac{\alpha}{x}}{\|\alpha\|^2}\alpha$, $x\in\R^N$. Fix a (reduced) root system $\mathcal{R}\subset \R^N\setminus\{0\}$ and let $\mathfrak{C}\subset O(N,\R)$ denote the Coxeter (or Weyl) group generated by the root reflections $r_\alpha$, $\alpha\in\mathcal{R}$. Furthermore let $k:\mathcal{R}\to\C$ be a fixed multiplicity function and write $k_\alpha:=k(\alpha)$ for $\alpha\in\mathcal{R}$. In the following we shall need the weight function $\vartheta_k(x)=\prod_{\alpha\in\mathcal{R}^+}\vert\brac{\alpha}{x}\vert^{2k_\alpha}$ defined on $\R^N$. For $\xi\in\C^n$ and a fixed multiplicity function $k$ define the 1st order Dunkl operators
\[T_\xi(k)f(x)=\partial_\xi f(x) + \sum_{\alpha\in\mathcal{R}^+}k_\alpha\brac{\alpha}{x}\frac{f(x)-f(r_\alpha x)}{\brac{\alpha}{x}},\quad f\in C^1(\R^N),\]
where $\partial_\xi$ is the directional derivative in the direction of $\xi$. It follows from the $\mathfrak{C}$-invariance of the multiplicity function that the definition of $T_\xi(k)$ is independent of the choice of positive system $\mathcal{R}^+$. These operators are homogeneous of degree $-1$ and have many convenient properties. Let $\left<k\right>=\sum_{\alpha\in\mathcal{R}^+}k_\alpha=
\frac{1}{2}\sum_{\alpha\in\mathcal{R}}k_\alpha$. Fix an orthonormal basis $\{\xi_1,\ldots,\xi_n\}$ of $(\R^N,\brac{\cdot}{\cdot})$, and write $T_j(k)=T_{\xi_j}(k)$ for short. The Dunkl--Laplacian $\Delta_k:=\sum_{j=1}^n T_j(k)^2$ can be written explicitly as
\[\Delta_k f(x)=\Delta f(x) + \sum_{\alpha\in\mathcal{R}^+}k_\alpha\Bigl(\frac{2\brac{\nabla f(x)}{\alpha}}{\brac{\alpha}{x}}-\|\alpha\|^2\frac{f(x)-f(r_\alpha x)}{\brac{\alpha}{x}^2}\Bigr),\]
where $\nabla$ denotes the usual gradient operator.

A $k$-harmonic polynomial of degree $m\in\N$ is a homogeneous polynomial $p$ on $\R^N$ of degree $m$ such that $\Delta_k p=0$. Let $\mathcal{H}_k^m(\R^N)$ denote the space of $k$-harmonic polynomials of degree $m$. Furthermore let $d\sigma$ denote the standard measure on the unit $N$-sphere $\mathbb{S}^{N-1}$ in $\R^N$, and let \[d_k=\Bigl(\int_{\mathbb{S}^{N-1}}\vartheta_k(w)\,d\sigma(w)\Bigr)^{-1}.\] 
In the case $k\equiv 0$ the number $\Frac{1}{d_k}$ is the volume of the unit sphere in $\R^N$. Then $L^2(\mathbb{S}^{N-1},\vartheta_k(w)d\sigma(w))$ is a Hilbert space with respect to the inner product
\[\brac{f}{g}_k=d_k\int_{\mathbb{S}^{n-1}}f(w)\overline{g(w)}\vartheta_k(w)\,d\sigma(w).\]

The function spaces $\mathcal{H}_k^m(\R^N)\vert_{\mathbb{S}^{N-1}}$, $m=0,1,\ldots$ are mutually orthogonal with respect to $\brac{\cdot}{\cdot}_k$, and 
\[L^2(\mathbb{S}^{N-1},\vartheta_k(w)d\sigma(w)) \cong \bigoplus_{m\in\N}\mathcal{H}_k^m(\R^N)\vert_{\mathbb{S}^{N-1}}.\]

\begin{definition}
Let $\vartheta_{k,a}(x):=\|x\|^{a-2}\vartheta_k(x)$. Define $L^p_{k,a}(\R^N)=L^p(\R^N,\vartheta_{k,a}(x)dx)$ and $d\mu_{k,a}(x)=\vartheta_{k,a}(x)dx$. The norm of a function $f\in L^p_{k,a}(\R^N)$ will be written $\|f\|_p$ if it is clear from the context that the reference measure is the weighted measure $\mu_{k,a}$, and $\|f\|_{L^p_{k,a}}$ otherwise.
\end{definition}

In standard polar coordinates on $\R^N$ it holds that 
\begin{equation}\label{eqn.polar-measure}
\vartheta_{k,a}(x)dx = r^{2\left<k\right>+N+a-3}\vartheta_k(w)\,dr\,d\sigma(w),\end{equation} implying the existence of a unitary isomorphism
\[L^2(\mathbb{S}^{N-1},\vartheta_k(w)\,d\sigma(w))\widehat{\otimes} L^2(\R_+,r^{2\left<k\right>+N+a-3}dr)\longrightarrow L^2_{k,a}(\R^N)\]
where $dx$ is the usual Lebesgue measure on $\R^N$. Hence we arrive at the very useful orthogonal decomposition
\[\bigoplus_{m\in\N}\bigl(\mathcal{H}_k^m(\R^N)\vert_{\mathbb{S}^{N-1}}\bigr) \otimes L^2(\R_+,r^{2\left<k\right>+N+a-3}dr)\overset{\simeq}{\longrightarrow}  L^2_{k,a}(\R^N).\]
Let $\lambda_{k,a,m}=\frac{1}{a}(2m+2\left<k\right>+N-2)$ and let
\[L^{(\lambda)}_\ell(t)=\frac{(\lambda+1)_\ell}{\ell!}\sum_{j=0}^\ell \frac{(-\ell)_j t^j}{(\lambda+1)_j j!} = \sum_{j=0}^\ell \frac{(-1)^j\Gamma(\lambda+\ell+1)}
{(\ell-j)!\Gamma(\lambda+j+1)} \frac{t^j}{j!}\]
denote the usual one-dimensional  Laguerre polynomial. For $x=rw\in\R^N$ (with $r>0$ and $w\in\mathbb{S}^{N-1}$), and $p\in\mathcal{H}_k^m(\R^N)$ define 
\[\Phi_\ell^{(a)}(p,x)= p(x)L_\ell^{(\lambda_{k,a,m})} 
\left(\tfrac{2}{a}\|x\|^2\right) \exp\left(-\tfrac{1}{a}\|x\|^a\right)
=p(w)r^mL_\ell^{(\lambda_{k,a,m})}\left(\tfrac{2}{a}r^a\right)\exp\left(-\tfrac{1}{a}r^a\right).\]

Furthermore let $W_{k,a}(\R^N)=\mathrm{span}_{\C} \{\Phi_\ell^{(a)}(p,\cdot)\,:\, \ell,m\in\N, p\in\mathcal{H}_k^m(\R^N)\}$; this subspace is dense in $L^2_{k,a}(\R^N)$ according to Proposition 3.12 in \cite{BSKO}.

For our later purposes the following result is clearly of importance; it appears as Corollary~3.2.2 in \cite{BSKO}.

\begin{corollary}
Let $a>0$ and $k$ be as above.
\begin{enumerate}
\item The differential-difference operator $\Delta_{k,a}=\|x\|^{2-a}\Delta_k-\|x\|^a$ is an essentially self-adjoint operator on $L^2_{k,a}(\R^N)$;
\item There is no continuous spectrum of $\Delta_{k,a}$;
\item The discrete spectrum of $-\Delta_{k,a}$ is given by
\[\begin{split}
& \{2a\ell+2m+2\left<k\right>+N-2+a\,:\,\ell,m\in\N\} \text{ if } N\geq 2\\
& \{2a\ell+2\left<k\right>+a\pm 1\,:\,\ell\in\N\}\text{ if } N=1.\end{split}\]
\end{enumerate}
\end{corollary}

The next result is \cite[Theorem~3.39]{BSKO}.

\begin{theorem}Assume $a>0$ and that the nonnegative multiplicity function $k$ satisfies the condition $a+2\left<k\right>+N-2>0$. Let $\C^+=\{z\in\C\,:\,\mathrm{Re }z>0\}$.
\begin{enumerate}
\item The map $\C^+\times L^2_{k,a}(\R^N)\to L^2_{k,a}(\R^N)$, $(z,f)\mapsto e^{-z\Delta_{k,a}}f$ is continuous.
\item For any $p\in\mathcal{H}_k^m(\R^N)$ and $\ell\in\N$ it holds that $e^{-z\Delta_{k,a}}\Phi_\ell^{(a)}(p,\cdot)=e^{-z(\lambda_{k,a,m}+2\ell+1)}\Phi_\ell^{(a)}(p,\cdot)$.
\item The operator norm of $e^{-z\Delta_{k,a}}$ equals $\exp(-\frac{1}{a}(2\left<k\right>+N+a-2)\mathrm{Re}(z))$.
\item If $\mathrm{Re}(z)>0$, then $e^{-z\Delta_{k,a}}$ is a Hilbert--Schmidt operator.
\item If $\mathrm{Re}(z)=0$, then $e^{-z\Delta_{k,a}}$ is a unitary operator.
\end{enumerate}
\end{theorem}
In particular the operator $e^{-z\Delta_{k,a}}$ has a distribution kernel $\Lambda_{k,a}(x,y;z)$ such that
\[e^{-z\Delta_{k,a}}f(x)=\int_{\R^N}\Lambda_{k,a}(x,y;z)\vartheta_{k,a}(y)\,dy\]
for $f\in L^2_{k,a}(\R^N)$.

In general no closed expression for $\Lambda_{k,a}(x,y;z)$ is available; the paper \cite{BSKO} lists explicit formulae whenever $N=1$ and $a>0$ is arbitrary, or whenever $N\geq 2$ is arbitrary and $a\in\{1,2\}$. We shall recall these below but for some applications it suffices to have a series expansion.

\begin{definition}
The \emph{$(k,a)$-generalized Fourier transform $\mathcal{F}_{k,a}$} is the unitary operator
\[\mathcal{F}_{k,a}= \exp\left[\frac{i\pi}{2}\left(\frac{1}{a}\bigl(2\left<k\right>+n+a-2\bigr)\right) \right] \exp\left[\frac{i\pi}{2a}\bigl(\|x\|^{2-a}\Delta_k-\|x\|^a\bigr)\right] \]
defined on $L^2_{k,a}(\R^N)$.
\end{definition}
Some notable special cases include:
\begin{itemize} 
\item $a=2, k\equiv 0$. Then $\mathcal{F}_{k,a}$ is the Euclidean Fourier transform (see \cite{Howe88});
\item $a=1, k\equiv 0$. Then $\mathcal{F}_{k,a}$ is the Hankel transform and appears in  \cite{Kobayashi-Mano} as the unitary inversion operator of the  Schr\"odinger model of the minimal representation of the group $\mathrm{O}(N+1,2)$.
\item $a=2, k>0$. Then we recover the Dunkl transform.
\end{itemize}

In other words $\mathcal{F}_{k,a}$ `interpolates' between several types of integral transforms and allows a unified study of these. 

\begin{theorem}\label{thm-plancherel}
Let $a>0$ be given and assume $k$ satisfies $a+2\left<k\right>+N>2$.
\begin{enumerate}
\item (Plancherel formula) The operator $\mathcal{F}_{k,a}$ is a unitary map of $L^2_{k,a}(\R^N)$ onto itself.
\item $\mathcal{F}_{k,a}(\Phi_\ell^{(a)}(p,\cdot))=e^{-i\pi(\ell+m/a)}\Phi_\ell^{(a)}(p,\cdot)$ for any $\ell,m\in\N$ and $p\in\mathcal{H}_k^m(\R^N)$.
\item $\mathcal{F}_{k,a}$ is of finite order if and only if $a\in\mathbb{Q}$. If $a\in\Q$ is of the form $a=\frac{q}{q^\prime}$, with $q$, $q^\prime$ positive, then $(\mathcal{F}_{k,a})^{2q}=\mathrm{Id}$. In particular $\mathcal{F}_{k,a}^{-1}=\mathcal{F}_{k,a}^{2q-1}$.
\end{enumerate}
\end{theorem}
\begin{proof}
See Theorem 5.1 in \cite{BSKO}. The last statement appears as \cite[Corollary~5.2]{BSKO}.
\end{proof}

\begin{theorem}[Inversion formula]\label{thm.inversion}
Let $k$ be a non-negative multiplicity function.
\begin{enumerate}[label=(\roman*)]
\item Let $r\in\N$ and suppose that $2\left<k\right>+N>2-\Frac{1}{r}$. Then $(\mathcal{F}_{k,1/r})^{-1}=\mathcal{F}_{k,1/r}$.
\item Let $r\in\N_0$ and suppose that $2\left<k\right>+N>2-\frac{2}{2r+1}$. Then $\mathcal{F}_{k,\frac{2}{2r+1}}$ is a unitary operator of order four on $L^2_{k,\frac{2}{2r+1}}(\R^N)$. The inversion formula is given as
\[\bigl(\mathcal{F}^{-1}_{k,\frac{2}{2r+1}}f\bigr)(x)=\bigl(\mathcal{F}_{k,\frac{2}{2r+1}}f\bigr)(-x).\]
\end{enumerate}
\end{theorem}
\begin{proof}
See Theorem 5.3 in \cite{BSKO}.
\end{proof}

By the Schwartz kernel theorem there exists a distribution kernel $B_{k,a}(\xi,x)$ such that
\[\mathcal{F}_{k,a}f(\xi)=c_{k,a}\int_{\R^N} B_{k,a}(\xi,x)f(x)\vartheta_{k,a}(x)\,dx.\]

We will need to introduce various special functions in order for this to be explicit. Consider the $I$-Bessel function $I_\lambda(w)=e^{-\frac{\pi}{2}i\lambda}J_\lambda(e^{\frac{\pi}{2}i}w)$, where $J_\lambda$ is the standard Bessel function. Moreover define
\begin{equation}\label{eqn.I-Bessel}
\wtilde{I}_\lambda(w):=(w/2)^{-\lambda}I_\lambda(w)=\frac{1}{\sqrt{\pi}\Gamma(\lambda+\frac{1}{2})}\int_{-1}^1e^{wt}(1-t^2)^{\lambda-\frac{1}{2}}\,dt.
\end{equation}

It follows by standard estimates for special functions that $|\widetilde{I}_\lambda(w)|\leq \Gamma(\lambda+1)^{-1}e^{|\Re w|}$,	in addition to the equally standard estimate
\begin{equation}\label{eqn.I-derivative}
|\widetilde{I}^\prime_\lambda(w)|\lesssim e^{|\Re w|}
\end{equation}
where the  constant implied in the notation is independent of $w$. An analogous estimate holds for higher derivatives $\widetilde{I}^{(\ell)}_\lambda$

\begin{example}[The case $N=1$, $a>0$]\label{example.B1d}
For $N=1$ there is but a single choice of root system, $\mathcal{R}=\{\pm 1\}$ (up to scaling), and $\mathfrak{C}=\{\mathrm{id},\sigma\}\simeq\Z/2\Z$, as well as $\left<k\right>=k>\frac{1}{2}(1-a)$.  In this case $\vartheta_{k,a}(x)=\vert x\vert^{2k+a-2}dx$, 
\[B_{k,a}(x,y)= \Gamma\Bigl(\frac{2k+a-1}{a}\Bigr) \Bigl[\wtilde{J}_{\frac{2k-1}{a}}\bigl(\tfrac{2}{a}\vert xy\vert^{a/2}\bigr) + \frac{xy}{(ia)^{2/a}}\wtilde{J}_{\frac{2k+1}{a}}\bigl(\tfrac{2}{a}\vert xy\vert^{a/2}\bigr)\Bigr],\]
where the branch of $i^{2/a}$ is chosen so that $1^{2/a}=1$, where $\wtilde{J}_\nu(w)=\wtilde{I}_\nu(-iw)$, and $\wtilde{I}_\nu(w)$ is the normalized Bessel function defined above. In addition, it follows from the aforementioned series expression for $\Lambda_{k,a}(x,y;z)$ in terms of the radial components $\Lambda_{k,a}^{(m)}(x,y;z)$, that

\begin{multline}\label{eqn.Lambda}
\Lambda_{k,a}(x,y;z) = \Gamma\Bigl(\frac{2k+a-1}{a}\Bigr)\frac{e^{-\frac{1}{a}(\vert x\vert^a+\vert y\vert^a)\coth z}}{(\sinh z)^{\frac{2k+a-1}{a}}}\\ \times \Bigl[\wtilde{I}_{\frac{2k-1}{a}}\Bigl(\frac{2}{a}\frac{\vert xy\vert^{a/2}}{\sinh z}\Bigr) + \frac{1}{a^{2/a}}\frac{xy}{(\sinh z)^{2/a}}\wtilde{I}_{\frac{2k+1}{a}}\Bigl(\frac{2}{a}\frac{\vert xy\vert^{a/2}}{\sinh z}\Bigr)\Bigr]
\end{multline}
\end{example}

\begin{example}[The case $N\geq 2$, $a\in\{1,2\}$]\label{example-a12}
Let $V_k$ denote the Dunkl intertwining operator associated with the given choice of root system and multiplicity function $k$. For a continuous function $h$ of one variable, set $h_y(\cdot)=h(\left<\cdot,y\right>)$ for $y\in\R^N$ and define $(\widetilde{V}_kh)(x,y)=(V_kh_y)(x)$. 
It is established in \cite[Secion~4.4]{BSKO} that for $z\in\mathbb{C}^+\setminus i\pi\mathbb{Z}$, $x=r\omega$ and $y=s\eta$ (polar coordinates in $\R^N$), one has the identity $\Lambda_{k,a}(x,y;z)=\widetilde{V}_k(h_{k,a}(r,s;z;\cdot))(\omega,\eta)$, where
\[\begin{split}
h_{k,a}(r,s;z,t)&=\frac{\exp[-\frac{1}{a}(r^a+s^a)\coth z]}{\sinh(z)^{(2\langle k\rangle+N+a-2)/a}}\\
&\times \begin{cases} \Gamma\bigl(\langle k\rangle+\frac{N-1}{2}\bigr)\widetilde{I}_{\langle k\rangle+\frac{N-3}{2}}\Bigl(\frac{\sqrt{2}(rs)^{1/2}}{\sinh z}(1+t)^{1/2}\Bigr)&\text{when } a=1\\
\exp\Bigl(\frac{rst}{\sinh z}\Bigr)&\text{when } a=2
\end{cases}
\end{split}\]
where $\widetilde{I}_\nu$ is the normalized $I$-Bessel function defined in \eqref{eqn.I-Bessel}.
\end{example}

\begin{lemma}\label{lemma.B-parameters} Assume $N\geq 1$, $k\geq 0$, $a+2\langle k\rangle+N>2$, and that \textbf{exactly one} of the following additional assumptions holds:
\begin{equation}\label{eqn.B-parameters}
\left\{\begin{aligned}
(i)&\quad N=1 \text{ and } a>0;\\
(ii)&\quad a\in\{1,2\};\\
(iii)&\quad k\equiv 0 \text{ and }a=\frac{2}{m} \text{ for some } m\in\N.
\end{aligned}\right.
\end{equation}
Then $B_{k,a}$ is uniformly bounded, that is, $|B_{k,a}(\xi,x)|\leq C$ for all $x,\xi\in\R^N$, where $C$ is a finite constant that only depends on $N$, $k$, and $a$.
\end{lemma}
\begin{proof}
The case $N=1$ follows from the explicit formula for $B_{k,a}$ in example \ref{example.B1d} or by observing that the one-dimensional Dunkl-kernel $B_{k,2}$ is known to be uniformly bounded by $1$. The kernel $B_{k,a}$ is a scaled version and is therefore uniformly bounded by a constant that depends on $a$.

The second case is stated as \cite[Theorem~5.11]{BSKO}, and the remaining case was established in \cite[Theorem~3]{deBie}.
\end{proof}

\begin{quote}
\textbf{Convention:} We shall replace $\mathcal{F}_{k,a}$ by the rescaled version $\mathcal{F}_{k,a}/C$ but continue to use the same symbol $\mathcal{F}_{k,a}$.
\end{quote}

It is presently unknown whether the kernel $B_{k,a}$ is uniformly bounded for all admissible parameters $a$, so the following Hausdorff--Young inequality -- which was not stated in \cite{BSKO} -- might not be valid in general. We list it here since it will be used in section \ref{section.entropy} where inequalities for Shannon entropy are obtained.

\begin{prop}\label{prop.HY}
Assume $N$, $k$, and $a$ meet the assumptions in lemma \ref{lemma.B-parameters}. Let $p\in[1,2]$ be fixed and set $p':=\frac{p}{p-1}$. Then  $\|\mathcal{F}_{k,a}f\|_{L^{p^\prime}_{k,a}}\leq \|f\|_{L^p_{k,a}}$ for all $f\in L^p_{k,a}(\R^N)$.
\end{prop}
Without the aforementioned convention in place one would have to include a constant on the right hand side due to interpolation. As this constant is a nuisance and tends to cloud later applications of the Hausdorff--Young inequality, we decided to rescale $\mathcal{F}_{k,a}$ to get rid of the interpolation constant.

\begin{proof}
Since $\mathcal{F}_{k,a}$ is unitary on $L^2_{k,a}(\R^N)$ according to theorem \ref{thm-plancherel}, it is of strong type $(2,2)$. Moreover $|\mathcal{F}_{k,a}f(\xi)|\leq \|f\|_1$ for every $\xi\in\R^N$ by convention and $f\in L^1_{k,a}(\R^N)$ by lemma \ref{eqn.B-parameters}, so $\mathcal{F}_{k,a}$ is of strong type $(1,\infty)$. The conclusion now follows from the Riesz--Thorin interpolation theorem.
\end{proof}
A more precise formulation is given as follows. Let $f\in(L^1_{k,a}\cap L^2_{k,a})(\R^N)$. If $1\leq p\leq 2$, $p^\prime=\frac{p}{p-1}$, then $\|\mathcal{F}_{k,a}f\|_{p^\prime}\leq c_p\|f\|_p$. Since $L^p$ is dense in $(L^1_{k,a}\cap L^2_{k,a})(\R^N)$ for $1\leq p\leq 2$, the transform $\mathscr{F}_pf$ can be defined uniquely for all $f\in L^p_{k,a}(\R^N)$, $1\leq p\leq 2$, so that $\mathscr{F}_p:L^p_{k,a}(\R^N)\to L^{p^\prime}_{k,a}(\R^N)$ is a linear contraction with $\mathscr{F}_pf=\mathcal{F}_{k,a}f$ for all $f\in (L^1_{k,a}\cap L^2_{k,a})(\R^N)$.

\begin{lemma}\label{lemma.no-surjective}
Assume $N$, $k$, and $a$ meet the assumptions in lemma \ref{lemma.B-parameters}, and let $p\in(1,2]$. The map $\mathscr{F}_p:L^p_{k,a}(\R^N)\to L^{p^\prime}_{k,a}(\R^N)$ is surjective if and only if $p=2$.
\end{lemma}
\begin{proof}
The `if'-part being the Plancherel theorem for $\mathcal{F}_{k,a}$, assume $p\in (1,2)$. The `only if'-part follows by contradiction as in the proof of \cite[Corollary~3.10]{Johansen-CO}, with the obvious notational modifications.
\end{proof}

A weighted extension will be established in theorem \ref{thm.weighted}, a special case of which will be an analogue of Pitt's inequality.

\begin{lemma}\label{lemma.p1p2}
Assume $N$, $k$, and $a$ meet the assumptions in lemma \ref{lemma.B-parameters}. If $f$ belongs to $(L^{p_1}_{k,a}\cap L^{p_2}_{k,a})(\R^N)$ for some $p_1,p_2\in[1,2]$, then $\mathscr{F}_{p_1}f=\mathscr{F}_{p_2}f$ $\mu_{k,a}$-almost everywhere on $\R^N$.
\end{lemma}
\begin{proof}
Choose a sequence $\{g_n\}_{n=1}^\infty$ of simple functions on $\R^N$ such that 
\[\lim_{n\to\infty}\|f-g_n\|_{p_1}=\lim_{n\to\infty}\|f-g_n\|_{p_2}=0.\]
Each function $\mathcal{F}_{k,a}g_n$ belongs to $(L^{p_1^\prime}_{k,a}\cap L^{p_2^\prime}_{k,a})(\R^N)$ by the Hausdorff--Young inequality, and
\[\lim_{n\to\infty}\|\mathscr{F}_{p_1}f-\mathcal{F}_{k,a}g_n\|_{p_1^\prime}=\lim_{n\to\infty}\|\mathscr{F}_{p_2}f-\mathcal{F}_{k,a}g_n\|_{p_2^\prime}=0.\]
One can therefore extract subsequences $\{\mathcal{F}_{k,a}g_{n_k}\}_{k=1}^\infty$ and $\{\mathcal{F}_{k,a}g_{n_l}\}_{l=1}^\infty$ of $\{\mathcal{F}g_n\}_{n=1}^\infty$ such that $\mathcal{F}_{k,a}g_{n_k}\to\mathscr{F}_{p_1}f$ and $\mathcal{F}_{k,a}g_{n_l}\to\mathscr{F}_{p_2}f$ $\mu_{k,a}$-almost everywhere on $\R^N$, from which it follows that $\mathscr{F}_{p_1}f=\mathscr{F}_{p_2}f$ $\mu_{k,a}$-almost everywhere on $\R^N$ as claimed.
\end{proof}

\begin{lemma}\label{lemma.schwartz}
Assume $N$, $k$, and $a$ satisfy either (i) or (ii) in lemma \ref{lemma.B-parameters}. The Euclidean Schwartz space $\mathscr{S}(\R^N)$ is dense in $L^p_{k,a}(\R^N)$ for $p\in[1,\infty)$ and invariant under $\mathcal{F}_{k,a}$.
\end{lemma}
\begin{proof}
Only the invariance under $\mathcal{F}_{k,a}$ needs to be addressed. In the case $a=2$, the statement is that $\mathcal{S}(\R^N)$ is invariant under the Dunkl transform, a fact that was established in \cite[Corollary~4.8]{deJeu-dunkl}. In the general one-dimensional case one can redo de Jeu's proof, especially the boundedness of derivatives of the Dunkl kernel in \cite[Corollary~3.7]{deJeu-dunkl} for the `deformed' kernel function $B_{k,a}$; it follows from the explicit formula in \ref{example.B1d} that it satisfies the same bounds, implying that the transform $\mathcal{F}_{k,a}$, $a>0$, leaves $\mathscr{S}(\R)$ invariant as well.
\medskip

The case $N\geq 2$, $a=1$, is also handled by a direct appeal to the explicit formula for $B_{k,1}$, this time in example \ref{example-a12}. The estimate for the derivatives of $B_{k,a}$, replacing \cite[Corollary~3.7]{deJeu-dunkl} or \cite[Corollary~5.4]{Rosler-duke}, is obtained from \eqref{eqn.I-derivative} and example \ref{example-a12}, together with the Bochner-type integral representation of the intertwining operator $\widetilde{V}_k$: One utilities that derivatives of $h_{k,a}(r,s;z,t)$ in the case $a=1$ -- both as a function of $r$ and as a function of $s$ -- grows exponentially at the same rate as derivatives of the kernel function $\exp(\frac{rst}{\sinh z})$ in the Dunkl-case $a=2$. The argument by de Jeu that leads from \cite[Corollary~3.7]{deJeu-dunkl} to \cite[Corollary~4.8]{deJeu-dunkl} can therefore be repeated.
\end{proof}

The remainder of the section is concerned with interpolation results in Lorentz spaces that will be needed in our proof of the Hardy--Littlewood inequality. The interested reader may consult \cite[Chapter~V]{Stein-Weiss-analysis} for detailed proofs and historical remarks. Let $(X,\mu)$ be a $\sigma$-finite measure space and let $p\in(1,\infty)$. Define
\[\|f\|^*_{p,q}=\begin{cases}\displaystyle \Bigl(\frac{q}{p}\int_0^\infty t^{q/p-1}f^*(t)^q\,dt\Bigr)^{1/q}&\text{if } q<\infty\\ 
\displaystyle \sup_{t>0} t\lambda_f(t)^{1/p}&\text{when } q=\infty\end{cases}\]
where $\lambda_f$ is the distribution function of $f$ and $f^*$ the non-increasing rearrangement of $f$, that is
\[\lambda_f(s) =\mu(\{x\in X\,:\, |f(x)|>s\}) \quad\text{and}\quad f^*(t)=\inf \{s\,:\,\lambda_f(s)\leq t\}.\]
By definition, the Lorentz space $L^{p,q}(X)$ consists of measurable functions $f$ on $X$ for which $\|f\|^*_{p,q}<\infty$. 

\begin{definition}
Let $(X,d\mu)$ and $(Y,d\overline{\nu})$ be $\sigma$-finite measure spaces. A linear operator $T:L^p(X,d\mu)\to L^q(Y,d\overline{\nu})$ is \emph{strong type $(p,q)$} if it is continuous on $L^p(X,d\mu)$. Moreover, $T$ is \emph{weak type $(p,q)$} if there exists a positive constant $K$ independent of $f$ such that for all $f\in L^p(X,d\mu)$ and all $t>0$,
\[\mu\bigl(\bigl\{y\in Y\,:\, |Tf(y)|>t\bigr\}\bigr)\leq \Bigl(\frac{K}{s}\|f\|_{L^p(X,d\mu)}\Bigr)^q.\]
The infimum if such $K$ is the weak type $(p,q)$ norm of $T$.
\end{definition}

Although $\|\cdot\|^*_{p,q}$ is merely a seminorm in general, the spaces $L^{p,q}(X)$ are very useful in interpolation arguments. The following interpolation theorem is classical and can be found as Theorem 3.15 in \cite[Chapter~V]{Stein-Weiss-analysis}. It subsumes the interpolation theorem of Marcinkiewicz, for example.

\begin{theorem}[Interpolation between Lorentz spaces]\label{thm.lorentz-inter} Suppose $T$ is a subadditive operator of (restricted) weak types $(r_j,p_j)$, $j=0,1$, with $r_0<r_1$ and $p_0\neq p_1$, then there exists a constant $B=B_\theta$ such that  $\|T\|^*_{p,q}\leq B\|f\|^*_{r,q}$ for all $f$ belonging to the domain of $T$ and to $L^{r,q}$, where $1\leq q\leq\infty$, 
\begin{equation}\label{eqn.interpolate-indices}
\frac{1}{p}=\frac{1-\theta}{p_0}+\frac{\theta}{p_1},\quad \frac{1}{r}=\frac{1-\theta}{r_0}+\frac{\theta}{r_1}\quad\text{and}\quad 0<\theta<1.
\end{equation}
\end{theorem}	
\begin{corollary}[Paley's extension of the Hausdorff--Young inequality]\label{cor.HY-Lorentz}
	If $f\in L^p(\R^n)$, $1<p\leq 2$, then its Fourier transform $\what{f}$ belongs to $L^{p',p}(\R^n)$ and there exists a constant $B=B_p$ independent of $f$ such that $\|\what{f}\|_{p',p}^*\leq B_p\|f\|_p$, where $\frac{1}{p}+\frac{1}{p'}=1$. In particular the Fourier transform is a continuous linear mapping from $L^p(\R^n)$ to the Lorentz space $L^{p',p}(\R^n)$ for $1<p<2$.
\end{corollary}
\begin{proof}
Taking $(r_0,p_0)=(1,\infty)$, $(r_1,p_1)=(2,2)$ in theorem \ref{thm.lorentz-inter}, the conditions in \eqref{eqn.interpolate-indices} translate into $\frac{1}{p}=\frac{\theta}{2}$ and $\frac{1}{r}=1-\frac{\theta}{2}$, that is, $r=p'$. Furthermore take $q=r$. Since $\theta\in(0,1)$ in the hypothesis of theorem \ref{eqn.interpolate-indices}, the role of $p$ and $p'$ must be exchanged when we consider the setup in the present corollary. (Since $\frac{2}{p}=\theta\in(0,1)$ if and only if $p>2$). With this adjustment in mind, the conclusion to theorem \ref{eqn.interpolate-indices} becomes
$\|\what{f}\|^*_{p',p}\leq B\|f\|^*_{p,p}=B\|f\|_p$.
\end{proof}

As in the proof of corollary \ref{cor.HY-Lorentz}, we obtain the following extension immediately from the interpolation theorem \ref{thm.lorentz-inter}.
\begin{corollary}Assume $N$, $k$, and $a$ meet the assumptions in lemma \ref{lemma.B-parameters}. The $(k,a)$-generalized transform $\mathcal{F}_{k,a}$ is a continuous mapping from $L^p_{k,a}(\R^N)$ to $L^{p',p}_{k,a}(\R^N)$ whenever $1<p<2$.
\end{corollary}

The preceding two corollaries are stronger than their respective standard forms since $L^{p^\prime,p}$ is continuously and properly embedded in $L^{p^\prime}$.

The last result on Lorentz spaces that we will need is due to R. O'Neil, \cite{Oneil-convolution}, and concerns the pointwise product of two functions. 

\begin{theorem}\label{thm.Oneil} 
	Let $q\in(2,\infty)$ and set $r=\frac{q}{q-2}$. For $g\in L^q(X)$ and $h\in L^{r,\infty}(X)$ it holds that $gh$ belongs to $L^{q',q}(X)$ with $\|gh\|^*_{q',q}\leq\|g\|_q\|h\|_{r,\infty}^*$.
\end{theorem}

\subsection{Hankel transforms and radial functions}\label{section.radial}
It is a very useful fact of classical analysis that the Fourier transform of a radial function on $\R^n$ is radial and given by a suitable Hankel transform of the radial projection. It was observed in Proposition~2.4 in \cite{Rosler-Voit-Markov} that the Dunkl transform of a radial function in $L^1(\R^n,\vartheta_k(x)dx)$ is also radial and expressed in terms of an appropriate Hankel transform. Specifically, if $f\in (L^1_{k,a}\cap L^2_{k,a})(\R^N)$ is of the form $f(x)=p(x)\psi(\|x\|)$ for some $p\in\mathcal{H}_k^m(\R^N)$ and some function $\psi$ on $\R_+$, then
\begin{equation}\label{eqn.Fourier-radial}
\mathcal{F}_{k,a}f(\xi)=a^{-((2m+2\left<k\right>+N-2)/a)} e^{-\frac{i\pi}{a}m}p(\xi)
\mathbf{H}_{a,\frac{2m+2\left<k\right>+N-2}{a}}(\psi)(\|\xi\|)
\end{equation}
where $\displaystyle \mathbf{H}_{a,\nu}(\psi)(s):=\int_0^\infty \psi(r)\widetilde{J}_{\nu}\Bigl(\frac{2}{a}(rs)^{a/2}\Bigr)r^{a(\nu+1)-1}\,dr$. Recall that
\[\widetilde{J}_\nu(\omega)=\Bigl(\frac{\omega}{2}\Bigr)^{-\nu}J_\nu(\omega)=\sum_{\ell=0}^\infty\frac{(-1)^\ell\omega^{2\ell}}{2^{2\ell}\ell!\Gamma(\nu+\ell+1)} = \frac{1}{\Gamma(\nu+1)}j_\nu(\omega),\]
where $j_\nu$ is the modified Bessel function that usually appears in the definition of the classical Hankel transform $\mathcal{H}_\nu$.

\begin{definition}
Given parameters $p\in[1,\infty)$, $a>0$, and $\nu>-1/2$, the norm $\|f\|_{p,a,\nu}$ of a measurable function $f$ on $\R_+$ is defined by
\[\|f\|_{p,a,\nu}=\Bigl(\int_0^\infty \vert f(r)\vert^p r^{a(\nu+1)-1}dr\Bigr)^{1/p}.\]
\end{definition}
Since the density 
$\vartheta_{k,a}(x)=\|x\|^{a-2}\prod_{\alpha\in\mathfrak{R}^+}\vert\brac{\alpha}{x}\vert^{2k_\alpha}=\|x\|^{a-2}\vartheta_k(x)$ is homogeneous of degree $2\left<k\right>+a-2$, it is clear that the $L^p$-norm of a radial function $f$ of the form $f(x)=\psi(\|x\|)$ can be expressed in terms of a suitable $L^p$-norm of $\psi$. This is seen by passing to polar coordinates $x=r\omega$, $r>0, \omega\in\mathbb{S}^{N-1}$ and we collect the precise statement for later reference in the following

\begin{lemma}\label{lemma.norm-radial} Let $1\leq p\leq 2$ and let $f\in L^p_{k,a}(\R^N)$ be radial of the form $f(x)=\psi(\|x\|)$ for a suitable measurable function $\psi$ on $\R_+$. Then
\begin{equation}\label{eqn.norm-radial}
\|f\|_{L^p_{k,a}} = K^{1/p}\|\psi\|_{L^p(\R_+,r^{2\left<k\right>+N+a-3}dr)} = K^{1/p}\|\psi\|_{p,a,\nu_a},
\end{equation} where $\nu_a:=\frac{2\left<k\right>+N-2}{a}$ and $K=K_{k,a,N}:=\int_{\mathbb{S}^{N-1}}\vartheta_{k,a}(\omega)d\sigma(\omega)$.
\end{lemma}
The parameter $\nu_a$ and the norm $\|\cdot\|_{p,a,\nu}$ are defined in such a way that we recover, in particular, the results of R\"osler and Voit (specifically Proposition 2.4 in \cite{Rosler-Voit-Markov}) by choosing $a=2$.

Two examples will be needed later so we collect them here:
\begin{enumerate}
\item The Gaussian $\gamma_t:y\mapsto e^{-t\|y\|^2}$. Here $\psi(y)=e^{-ty^2}$, with
\[\|\psi\|^p_{L^p(\R_+,r^{2\left<k\right>+N+a-3}dr)} = \frac{\Gamma(\frac{2\left<k\right>+N+a-2}{2})}{2(pt)^{\frac{2\left<k\right>+N+a-2}{2}}},\]
so lemma \ref{lemma.norm-radial} implies that
\begin{equation}\label{eqn.norm-gaussian}
\|\gamma_t\|_{L^p_{k,a}}=K^{1/p}\Bigl(\frac{\Gamma(\nu_a+1)}{2p^{\nu_a+1}}\Bigr)^{1/p}t^{-\frac{\nu_a+1}{p}},\end{equation}
where $\nu_a=\frac{2\left<k\right>+N-2}{a}$.
\item The function $g_\alpha:y\mapsto \vert y\vert^{-\alpha}\mathbf{1}_{B_r}(y)$, $0<\alpha<\frac{2\left<k\right>+N+a-2}{p}$, where $B_r=\{x\in\R^N\,:\,\vert x\vert\leq r\}$. Here $\psi(y)=y^{-\alpha}\mathbf{1}_{[0,r]}(y)$, $r>0$, and
\[\|\psi\|^p_{L^p(\R_+,r^{2\left<k\right>+N+a-3}dr)} =\frac{1}{2\left<k\right>+N+a-2-\alpha p}r^{2\left<k\right>+N+a-2-\alpha p}\]
so lemma \ref{lemma.norm-radial} implies that
\begin{equation}\label{eqn.norm-f-alpha}
\|g_\alpha\|_{L^p_{k,a}}=\frac{K^{1/p}}{(a(\nu_a+1)-\alpha p)^{1/p}}r^{\frac{a(\nu_a+1)}{p}-\alpha}.\end{equation}
\end{enumerate}

It follows from \eqref{eqn.Fourier-radial} and Lemma \ref{lemma.norm-radial} that $\|\mathcal{F}_{k,a}(f)\|_{L^p_{k,a}}= a^{-\nu_a}\|\mathbf{H}_{a,\nu_a}(\psi)\|_{p,a,\nu_a}$. Therefore
$\|\mathcal{F}_{k,a}(f)\|_{L^q_{k,a}}= K^{1/q}a^{-\nu_a}M_{a,q}\|\mathbf{H}_{2,\nu_a}(\widetilde{\psi})\|_{q,\nu_a}$, where $\widetilde{\psi}(u)=\psi((\frac{a}{2})^{1/a}u^{2/a})$, 
so a sharp Hausdorff--Young theorem for the Hankel transform $\mathbf{H}_{2,\nu_a}$ would imply a sharp Hausdorff--Young inequality for the restriction of $\mathcal{F}_{k,a}$ to radial functions. This seems to be an open problem, however.

\section{Further remarks on the Hausdorff--Young inequality}
The Hausdorff--Young inequality for $\mathcal{F}_{k,a}$ easily followed from general mapping properties and interpolation but the argument left out the possibility of a Hausdorff--Young inequality for $p>2$. It is a classical fact the Euclidean Fourier transform does not allow a Hausdorff--Young inequality for $L^p$-functions when $p>2$, and an explicit counterexample for the Fourier transform can be found in \cite[Section~4.11]{Titchmarsh}. We have been unable to find any such statement for the Dunkl transform $\mathcal{F}_{k,2}$, so we have included the following short section to settle the matter, as it fits nicely into the general theme of (weighted) inequalities for $\mathcal{F}_{k,a}$.

In this section only, the parameter $a$ is assumed to be chosen in such a way that the Hausdorff--Young inequality and the inversion formula for $\mathcal{F}_{k,a}$ are both valid. Comparing with theorem \ref{thm.inversion}, the parameters $N$, $k$, and $a$ must therefore satisfy one of the conditions
\begin{enumerate}[label=(\alph*)]
\item $N=1$, $k\geq 0$, $a>0$, and $a+2k+1>2$;
\item $N\geq 1$, $k\geq 0$, $a+2\langle k\rangle+N>2$, and $a=\frac{1}{r}$ for some $r\in\N$;
\item $N\geq 1$, $k\geq 0$, $a+2\langle k\rangle+N>2$, and $a=\frac{2}{2r+1}$ for some $r\in\N_0$.
\end{enumerate}
Let 
\[\sigma_ah(x)=\begin{cases} h(x)&\text{if (a) or (b) holds}\\
h(-x)&\text{if (c) holds}\end{cases}.\]
Recall that the case (c) subsumes the standard Dunkl transform (corresponding to the particular choice $r=0$). The inversion formula in theorem \ref{thm.inversion} can then be written succinctly as $\mathcal{F}_{k,a}^{-1}=\sigma_a\circ\mathcal{F}_{k,a}$. Since the Hausdorff--Young inequality is also required to hold, this range of permissible parameters is considerably narrower, however, as we must \emph{additionally} assume that the assumptions in lemma \ref{lemma.B-parameters} hold. It leaves us with the three cases considered in the following result, the proof of which is adapted from \cite{Chatterji}, where further historical remarks may be found.

\begin{prop}\label{prop.no-HY}
Assume $a+2\langle k\rangle+N>2$, and that either
\begin{itemize}
\item $N=1$ and $a>0$ (no further constraints),
\item $N\geq 1$ and $a\in\{1,2\}$, or
\item $k\equiv 0$ and $a=2/m$ for some $m\in\N$.
\end{itemize}
Let $p>2$ be fixed and $\mathcal{D}$ an $L^p$-dense subspace of $(L^1_{k,a}\cap L^p_{k,a})(\R^N)$. Then there exists no \emph{finite} constant $D_p$ such that the inequality $\|\mathcal{F}_{k,a}f\|_{p'}\leq D_p\|f\|_p$ holds for \emph{all} $f\in\mathcal{D}$. 
\end{prop}
\begin{proof}
Assume the conclusion is false and let $D_p$ be such a finite constant. Then there exists a continuous linear mapping $T:L^p_{k,a}(\R^N)\to L^{p^\prime}_{k,a}(\R^N)$ such that $Tf=\mathcal{F}_{k,a}f$ for all $f\in\mathcal{D}$. Since $1<p^\prime<2$ it follows from the Hausdorff--Young inequality and from $\sigma_a$ being an $L^{p^\prime}$-isometry that $\|\sigma_a\circ\mathcal{F}_{k,a}f\|_{p}=\|\mathcal{F}_{k,a}f\|_{p}\leq c_{p^\prime}\|f\|_{p^\prime}$ for all $f\in\mathcal{D}$. Hence  there exists a linear contraction $S:L^{p^\prime}_{k,a}(\R^N)\to L^p_{k,a}(\R^N)$ such that $Sf=\sigma_a\circ\mathscr{F}_{p^\prime}f$, where $\mathscr{F}_{p^\prime}f$ designates the $(k,a)$-generalized transform of the $L^{p^\prime}$-function $f$, whose existence is guaranteed by the Hausdorff--Young inequality.

Note that a function $f\in\mathcal{D}$ automatically belongs to $L^2_{k,a}(\R^N)$ (since $p>2$) and to $L^1_{k,a}(\R^N)$ by assumption. Therefore $Tf$ belongs to $(L^2_{k,a}\cap L^{p^\prime}_{k,a})(\R^N)$ whenever $f\in\mathcal{D}$. In particular the inversion and Plancherel formulae hold for $f$, implying that $S(Tf)=S(\mathscr{F}_2f)=\sigma_a\circ\mathscr{F}_{p^\prime}=\sigma_a\circ\mathscr{F}_2(\mathscr{F}_2f)=f$ for all $f\in\mathcal{D}$. Here it was used that $\mathscr{F}_{p^\prime}(\mathscr{F}_2f)=\mathscr{F}_2(\mathscr{F}_2f)$ $\mu_{k,a}$-almost everywhere on $\R^N$ since $\mathscr{F}_2f\in (L^2_{k,a}\cap L^{p^\prime}_{k,a})(\R^N)$, according to lemma \ref{lemma.p1p2}.

Since $S\circ T$ is continuous and $\mathcal{D}$ dense in $L^p_{k,a}(\R^N)$, it follows that $S\circ T=\mathrm{id}$ on $L^p_{k,a}(\R^N)$, in particular that $S$ is a left-inverse to $T$ and therefore surjective. This would imply that $\mathscr{F}_{p^\prime}$ were to be surjective on $L^{p^\prime}_{k,a}(\R^N)$, which -- according to lemma \ref{lemma.no-surjective} -- it is not. We have therefore arrived at a contradiction, proving the claim.
\end{proof}

The Hausdorff--Young inequality facilitates an extension of $\mathcal{F}_{k,a}$ to a continuous map $\mathscr{F}_p$ from $L^p_{k,a}(\R^N)$ into $L^{p^\prime}_{k,a}(\R^N)$, where $p\in(1,2)$ and $p^\prime=\frac{p}{p-1}$. By the same reasoning, one obtains a Hausdorff--Young inequality for the inverse transform $\mathcal{F}_{k,a}^{-1}$, giving rise to a continuous map $\mathscr{I}_q:L^q_{k,a}(\R^N)\to L^{q^\prime}_{k,a}(\R^N)$ that coincides with $\mathcal{F}_{k,a}^{-1}$ on $L^2_{k,a}(\R^N)$. This raises the 

\begin{quote}
\textbf{Question:} Do $\mathscr{F}_p$ and $\mathscr{I}_{p^\prime}$ coincide?
\end{quote}

We recently answered the analogous question for the one-dimensional Cherednik--Opdam transform in the affirmative, cf. \cite[Theorem~3.9]{Johansen-CO}, but the proof relied heavily on having a suitable convolution structure. While such a convolution is available in the Dunkl-case $a=2$, we have not yet investigated these matters in detail. It would be interesting to develop a strategy of proof that would also subsume the case $a=1$, at least.

It should also be noted that the proof of proposition \ref{prop.no-HY} uses in an essential way the special form of the inversion formula for $\mathcal{F}_{k,a}$ when $a\in\Q$, namely that $\mathcal{F}_{k,a}^{-1}=\sigma_a\circ\mathcal{F}_{k,a}$. This rules out an immediate extension to \emph{arbitrary} deformation parameters $a\in\Q$, since repeated application of the Hausdorff--Young inequality to higher iterates $\mathcal{F}_{k,a}^{2q-1}$ would break down except when $p=2$. For the proof it was very convenient, yet perhaps not essential, that the underlying measure space $(\R^N,\vartheta_{k,a})$ remained the same. One would otherwise have to show separately that $S$ cannot be surjective from $L^{p^\prime}$ onto $L^p$ (which again does not rely on special properties of the transform $\mathcal{F}_{k,a}$ or $\mathcal{T}$, as long as $S$ is injective). 

On the other hand the same methodology appears to be applicable in even dimensions to the Clifford--Fourier transform from  \cite{BrackX}, \cite{DeBie-Xu} although one would first have to establish a suitable Hausdorff--Young inequality. We intend to return to these matters elsewhere

\section{Hardy--Littlewood inequalities}
The classical Hausdorff--Young inequality $\|\what{f}\|_q\leq c_p\|f\|_p$, $1\leq p\leq 2$, $\frac{1}{p}+\frac{1}{q}=1$, for the Euclidean Fourier transform can be viewed as a partial extension of the Plancherel theorem to $L^p$-functions. More generally, the Fourier transform extends to a continuous mapping from $L^p(\R^N)$ into the Lorentz space $L^{p',p}(\R^N)$, a result that is due to Paley. Several variations on this theme were investigated by Hardy and Littlewood in \cite{Hardy-Littlewood} (Theorem 2, 3, 5 and 6, starting on page 175) who studied one-dimensional Fourier series anf the relation between summability properties of Fourier series and $\ell^p$-integrability of the sequence of Fourier coefficients. They were motivated by the following question: Suppose that $r>1$ and that $f$ and $|f|^r$ are integrable, where $f$ is a measurable $2\pi$-periodic function on $\R$. For what values of $s$ and $\kappa$ does it follow that the series $\sum_n n^{-\kappa}|a_n|^s$ is convergent?

For the Fourier transform in $\R^N$, their results can be stated as follows.  Fix $q\geq 2$ and let $f$ be a measurable function on $\R^N$ such that $x\mapsto f(x)\|x\|^{N(1-2/q)}$ belongs to $L^q(\R^N)$. Then $f$ has a well-defined Fourier transform $\what{f}$ in $L^q(\R^N)$ and there exists a positive constant $A_q$ independent of $f$ such that
\begin{equation}\label{HL-ineq1}
\Bigl(\int_{\R^N}\vert\what{f}(\xi)\vert^qd\xi\Bigr)^{1/q}\leq A_q\Bigl(\int_{\R^N}\vert f(x)\vert^q\|x\|^{N(q-2)}dx\Bigr)^{1/q}.
\end{equation}

A natural counterpart is to consider weights on the Fourier transform side. For every $p\in(1,2)$ there exists a positive constant $B_p$ independent of $f$ such that
\begin{equation}\label{HL-ineq2}
\Bigl(\int_{\R^N}|\what{f}(\xi)|^p|\xi|^{N(p-2)}\,d\xi\Bigr)^{1/p}\leq B_p\Bigl(\int_{\R^N}|f(x)|^p\,dx\Bigr)^{1/p}.
\end{equation}
Note that these inequalities do not involve the dual exponent $p^\prime$.  Our first result is a generalization of \eqref{HL-ineq2}; it generalizes the analogous statement  \cite[Lemma~4.1]{Anker-besov} for the Dunkl transform. For further historical remarks and extensions see \cite[Remark~6]{Benedetto-Heinig}. Analogous results in spherical harmonic analysis on Riemannian symmetric spaces were established in \cite{Eguchi}. 
\begin{prop}
Assume $a+2\langle k\rangle+N>2$. If $f\in L^p_{k,a}(\R^N)$ for some $p\in (1,2)$, then
\[\Bigl(\int_{\R^N}\|\xi\|^{2(\langle k\rangle+\frac{N+a-2}{2})(p-2)}|\mathcal{F}_{k,a}f(\xi)|^p\,d\mu_{k,a}(\xi)\Bigr)^{1/p}\leq C_p\Bigl(\int_{\R^N}|f(x)|^p\,d\mu_{k,a}(x)\,dx\Bigr)^{1/p}.\]
\end{prop}
\begin{proof}
Consider measure spaces $(X,d\overline{\mu})$ and $(Y,d\overline{\nu})$, where $X=Y=\R^N$, $d\overline{\mu}(x)=\vartheta_{k,a}(x)\,dx$, and $d\overline{\nu}(\xi)=\|\xi\|^{-4(\langle k\rangle+\frac{N+a-2}{2})}\vartheta_{k,a}(\xi)\,d\xi$, with $x,\xi\in\R^N$. Moreover define an operator $T$ on $L^2_{k,a}(\R^N)$ by
\[Tf(\xi)=\|\xi\|^{2(\langle k\rangle+\frac{N+a-2}{2})}\mathcal{F}_{k,a}f(\xi),\quad \xi\in\R^N.\]
Then $T$ is of strong type $(2,2)$ as an operator acting between Lebesgue spaces on $(X,d\overline{\mu})$ and $(Y,d\overline{\nu})$, since
\[\|Tf\|_{L^2(\overline{\nu})}^2 =\int_{\R^N}|Tf(\xi)|^2\|\xi\|^{-4(\langle k\rangle+\frac{N+a-2}{2})}\vartheta_{k,a}(\xi)\,d\xi = \int_{\R^N}|\mathcal{F}_{k,a}f(\xi)|^2\vartheta_{k,a}(\xi)\,d\xi =\|f\|_{L^2(\overline{\mu})},\]
by the Plancherel theorem \ref{thm-plancherel}.

The operator $T$ is furthermore of weak type $(1,1)$, which finishes the proof by an application of the Marcinkiewicz interpolation theorem. To verify this claim, let $t>0$ and $f\in L^1_{k,a}(\R^N)\setminus\{0\}$  be fixed, and define sets
\[A_t(f)=\{\xi\in\R^N\,:\,|Tf(\xi)|>t\}\quad\text{and}\quad
E_t(f)=\{\xi\in\R^N\,;\, \|\xi\|^{2(\langle k\rangle+\frac{N+a-2}{2})}>t/\|f\|_1\}\]
It follows from the basic inequality $\|\mathcal{F}_{k,a}f\|_\infty\leq\|f\|_1$ already used to establish the Hausdorff--Young inequality that $A_t(f)\subset E_t(f)$. Correspondingly, by passing to polar coordinates,
\[\begin{split}
\overline{\nu}(A_t(f)) &=\int_{A_t(f)} \frac{\vartheta_{k,a}(\xi)}{\|\xi\|^{4(\langle k\rangle+\frac{N+a-2}{2})}}\,d\xi \leq \int_{E_t(f)} \frac{\vartheta_{k,a}(\xi)}{\|\xi\|^{4(\langle k\rangle+\frac{N+a-2}{2})}}\,d\xi\\
&\lesssim \int_{a_t}^\infty \frac{r^{2(\langle k\rangle+\frac{N+a-2}{2})-1}}{r^{4(\langle k\rangle+\frac{N+a-2}{2})}}\,dr\quad\text{where }a_t=(t/\|f\|_1)^{\frac{1}{2(\langle k\rangle+\frac{N+a-2}{2})}}\\
&=c''a_t^{-2(\langle k\rangle+\frac{N+a-2}{2})}=c''\frac{\|f\|_1}{t}
\end{split}.\]
\end{proof}
\begin{remark}
An advantage in the above interpolation argument is that the possible lack of information on the integral kernel of $\mathcal{F}_{k,a}$ is not an issue. Instead one has to compensate by adding power weights.
\end{remark}

Two types of improvement can be obtained by using the more refined interpolation theorem between Lorentz spaces, theorem \ref{thm.lorentz-inter}: Inequalities with weights more general than the norm power $\|\cdot\|^{-4(\langle k\rangle+\frac{N+a-2}{2})}$ can be obtained and the permissible range of exponents $p$ can be enlarged. An efficient approach to both is to introduce the following terminology.

\begin{definition}If $\mu$ is any (positive Radon) measure on $\R^N$, a \emph{Young function (relative to $\mu$)} is a measurable function $\psi:\R^N\to\R$ with the property that $\mu(\{x\in\R^N\,:\,\vert\psi(x)\vert\leq t\})\lesssim t$ for all $t>0$. Given such a $\psi$, we let $L^{(p)}_{\psi}(\R^N)$, $2<p<\infty$, denote the Orlicz-type space of measurable functions $f$ on $\R^N$ for which
\[\|f\|_{(p),\psi}:=\Bigl(\int_{\R^N}|f(x)|^p|\psi(x)|^{p-2}\,d\mu(x)\Bigr)^{1/p}<\infty.\]
In other words $f$ belongs to $L^{(p)}_\psi(\R^N)$ if and only if $f\psi^{1-\frac{2}{p}}$ belongs to $L^p(\R^N,\mu)$.
\end{definition}
The choice of measure $\mu$ is determined by the relevant setting and will always be absolutely continuous with respect to Lebesgue measure on $\R^N$, the point of $\psi$ being that it allows for an easy proof of weak type $(1,1)$ estimates that are needed for the interpolation arguments. This will become clear in due time, firstly we wish to mention examples of Young functions.

\begin{example}\label{example.young}
\begin{enumerate}[label=(\roman*)]
\item	In $\R^N$, the function $\psi:x\mapsto\|x\|^m$ is a Young function with respect to Lebesgue measure if and only if $m=N$, since $|\{x\in\R^N\,:\,\|x\|^m<t\}|=|B(0,t^{1/m})|=Ct^{N/m}$. Since norms on $\R^N$ are equivalent, the $N$.th power of \emph{any} norm on $\R^N$ gives rise to a Young function.
\item More generally, $\psi:x\mapsto\|x\|^m$ is a Young function in $\R^N$ with respect to the weighted measure $d\mu_{k,a}(x)=\vartheta_{k,a}(x)\,dx$ for a unique choice of $m$. Since 
\[\mu_{k,a}(\{x\in\R^N\,:\,\psi(x)\leq t\}) = \mu_{k,a}(B(0,t^{1/m})) = Ct^{\frac{2\langle k\rangle+N+a-2}{m}}\]
it follows that $\psi$ is a Young function if and only if $m=2\langle k\rangle+N+a-2$.
\item The function $\vartheta_{k,a}$ itself is a Young function in $\R^N$ with respect to the weighted measure $\mu_{k,a}$, since
\[\mu_{k,a}(\{x\in\R^N\,:\,\vartheta_{k,a}(x)\leq t\}) \leq t\int_{\{\vartheta_{k,a}(x)\leq t\}}dx\]
where the integral is finite since the set $\{\vartheta_{k,a}(x)\leq t\}$ is compact (the level set is closed, and moreover bounded since $\vartheta_{k,a}(x)\asymp \|x\|^{2-a}\|x\|^{2\langle k\rangle}$, implying that $\{\vartheta_{k,a}(x)\leq t\}$ is contained in a ball $B(0,ct^{\frac{1}{2\langle k\rangle+a-2}})$ for some finite constant $c$).
\end{enumerate}
\end{example}

\begin{theorem}\label{thm.HY-qlarge}
	Let $q>2$ and $f\in L^{(q)}_{\psi}(\R^N)$, where $\psi$ is a Young function relative to $\mu_{k,a}$. There exists a positive constant $D_q$ independent of $f$ such that
	\[\int_{\R^N}|\mathcal{F}_{k,a}f(\xi)|^q\,d\mu_{k,a}(\xi)\leq D_q^q\|f\|^q_{(q),\psi}.\]
\end{theorem}
\begin{proof}
	Let $f$ be a simple function on $A$ and let $Tf(\lambda)=\mathcal{F}f(\lambda)$ (we do not need to add weights to the operator that enters the interpolation argument). Then $\|Tf\|^*_{\infty,\infty}=\|Tf\|_\infty\leq C\|f\|_1=\|f\|^*_{1,1}$, and by the Plancherel theorem it furthermore holds that $\|Tf\|_{2,\infty}^*\leq\|Tf\|^*_{2,2}\leq\|f\|_2\leq\|f\|^*_{2,1}$. By interpolation (cf. theorem \ref{thm.lorentz-inter}) it follows that $\|Tf\|^*_{q,q}\leq \|f\|^*_{q',q}$.
	
Now define $g(x)=f(x)\psi(x)^{1-\frac{2}{q}}$; then $g$ belongs to $L^q_{k,a}(\R^N)$ by hypothesis, since
\[\|g\|^q_{L^q_{k,a}}=\int_{\R^N}\vert f(x)\vert^q\vert\psi(x)\vert^{q-2}\,d\mu_{k,a}(x) = \|f\|^q_{(q),\psi}\]
It follows from the sublevel set estimate implied by $\psi$ being a Young function that 
\[\mu_{k,a}\bigl(\{x\in\R^N\,:\,\vert\psi(x)\vert^{\frac{2}{q}-1}>t\}\bigr) = \mu_{k,a}\Bigl(\Bigl\{x\in\R^N\,:\,\vert\psi(x)\vert^{1-\frac{2}{q}}<\frac{1}{t}\Bigr\}\Bigr)\leq Ct^{-\frac{q}{q-2}},\]
	whence $\psi^{\frac{2}{q}-1}$ belongs to $L^{r,\infty}_{k,a}(\R^N)$, where $r=\frac{q}{q-2}$. By an application of O'Neil's theorem \ref{thm.Oneil} it is seen that
\[\begin{split}
\int_{\R^N}|\mathcal{F}_{k,a}f(\xi)|^q\,d\mu_{k,a}(\xi)&\leq \|f\|^*_{p,q}\leq \|g\|_q\|\psi_h\|^*_{t,\infty}\\ &\leq C\int_{\R^N}|f(x)|^q|\psi(x)|^{q-2}\,d\mu_{k,a}(x) = c\|f\|^q_{(q),\psi}
\end{split}\]
which was the desired conclusion for simple functions. The extension to general functions in $L^{(q)}_{\psi}(\R^N)$ now follows by standard density arguments.
\end{proof}

The Dunkl-version of the following second version of the Hardy--Littlewood inequality was recently established in \cite[Proposition~4.3]{Johansen-HY} where the connection between a `flat' Heckman--Opdam transform and the (symmetrized) Dunkl transform was noted. The motivation behind this improvement involves several several intermediate results for the spherical Fourier transform on a Riemannian symmetric spaces that need not be repeated here. 

\begin{theorem}\label{prop.Dunkl-RS}
Let $1<q\leq 2$ be fixed. For $f\in L^p_{k,a}(\R^N)$ with $1<p\leq q$ there exists a finite constant $C_{p,q}$ independent of $f$ such that 
\[\Bigl(\int_{\R^N}|\mathcal{F}_{k,a}f(\xi)|^r(\|\xi\|\vartheta_{k,a}(\xi))^{r/p'-1}\,d\mu_{k,a}(\xi)\Bigr)^{1/r}\leq C_{p,q}\|f\|_{L^p_{k,a}}\]
where $\frac{1}{r}=1-\frac{q'-1}{p'}$.
\end{theorem}
\begin{proof}[Outline of proof]
Consider the measure spaces $(\R^N,d\mu_{k,a})$ and $(\R^N,d\overline{\mu})$, where $d\overline{\mu}(x)=\|x\|^{-Nq}\vartheta_{k,a}(x)^{1-Nq}\,dx$. Define $Tf(\xi)=|\mathcal{F}_{k,a}f(\xi)|(\|\xi\|\vartheta_{k,a}(\xi))^{\frac{Nq}{q^\prime}}$. Then $T$ is of strong type $(q,q^\prime)$ and of weak type $(1,1)$, the latter following from the estimate $\|\xi\|\vartheta_{k,a}(\xi)\asymp C\|\xi\|^{3-a+2\langle k\rangle}$ as in the proof of \cite[Theorem~3.10(i); Proposition~4.3]{Johansen-HY}.
\end{proof}

\begin{remark}
We obtain a more familiar form of the Hardy--Littlewood inequality in theorem \ref{prop.Dunkl-RS} by choosing as Young function a power of the Euclidean norm instead of the density $\vartheta_{k,a}$, that is $\psi(x)=\|x\|^{2\langle k\rangle+N+a-2}$,  cf. example \ref{example.young}(ii). The space $L^{(p)}_\psi(\R^N)$ now consists of all measurable functions $f:\R^N\to\C$ for which
\[\|f\|_{(p),\psi}=\Bigl(\int_{\R^N}|f(x)|^p\|x\|^{(p-2)(2\langle k\rangle+N+a-2)}\, d\mu_{k,a}(x)\Bigr)^{1/p}<\infty.\]
For $f\in L^{(p)}_\psi(\R^N)$ with $2\leq p<\infty$ it holds that
\[\Bigl(\int_{\R^N}|\mathcal{F}_{k,a}f(\xi)|^p\,d\mu_{k,a}(\xi)\Bigr)^{1/p}\leq C_p\Bigl(\int_{\R^N}|f(x)|^p\|x\|^{(2\langle k\rangle+N+a-2)(p-2)}\,d\mu_{k,a}(x)\Bigr)^{1/p},\]
which is the `dual' form of the Hardy--Littlewood inequality for the Dunkl transform obtained in \cite[Lemma~4.1]{Anker-besov}.
\end{remark}

\begin{remark}
It is briefly indicated in \cite[Remark~6, p.34]{Benedetto-Heinig} that some versions of the Hardy--Littlewood inequality can be obtained by another interpolation result in \cite{Benedetto-Heinig}. Since we have not provided detailed proofs of these interpolation results in the appendix, we found it appropriate to present a more direct, elementary proof with complete details.
\end{remark}

\section{Around the Heisenberg--Pauli--Weyl inequality}
The following uncertainty principle appeared as Theorem 5.29 in \cite{BSKO}:
\begin{theorem}\label{thm.Heisenberg-BSKO} Assume $N\geq 1$, $k\geq 0$, and that $a>0$ satisfies $a+2\left<k\right>+N>2$. For all $f\in L^2_{k,a}(\R^N)$, the $(k,a)$-generalized Fourier transform $\mathcal{F}_{k,a}$ satisfies the $L^2$-Heisenberg inequality
\[\bigl\| \|\cdot\|^{a/2}f\bigr\|_{L^2_{k,a}} \bigl\| \|\cdot\|^{a/2}\mathcal{F}_{k,a}f\bigr\|_{L^2_{k,a}} \geq \Bigl(\frac{2\left<k\right>+N+a-2}{2}\Bigr)\|f\|^2_{L^2_{k,a}}.\]
The inequality is saturated by functions of the form $f(x)=\lambda\exp(-c\|x\|^a)$ for some $\lambda\in\C$, $c>0$.
\end{theorem}
\begin{remark}\label{remark-scaling}
The exponent in the power weight $\|\cdot\|^{a/2}$ comes from simple scaling. An immediate advantage of the weighted interpolation techniques is that weights with different exponents can be used.
\end{remark}

The proof is elementary and based on spectral methods. A straightforward extension of their Heisenberg inequality is summarized in the following

\begin{prop}\label{prop.HPW}
Assume $N\geq 1$, $k\geq 0$, that $a>0$ satisfies $a+2\left<k\right>+N>2$, and that $\alpha,\beta\geq 1$. For every $f\in L^2_{k,a}(\R^N)$ it follows that
\[\bigl\| \|\cdot\|^{\alpha\cdot\frac{a}{2}}f\bigr\|_{L^2_{k,a}}^{\frac{\beta}{\alpha+\beta}} \bigl\|\|\cdot\|^{\beta\cdot\frac{a}{2}}\mathcal{F}_{k,a}f\bigr\|_{L^2_{k,a}}^{\frac{\alpha}{\alpha+\beta}}\geq \Bigl(\frac{2\left<k\right>+N+a-2}{2}\Bigr)^{\frac{\alpha\beta}{\alpha+\beta}}\|f\|_{L^2_{k,a}}.\]
\end{prop}
Note that the $L^2$-norm of $f$ on the right hand side is not squared; this is due to scaling and homogeneity but can also be explained heuristically by `counting' norm powers in the left hand side of the inequality: Indeed, $\|f\|$ appears raised to the power $\frac{\beta}{\alpha+\beta}+\frac{\alpha}{\alpha+\beta}=1$.
\begin{proof}
For $\alpha>1$ fixed and $\alpha'$ such that $\frac{1}{\alpha}+\frac{1}{\alpha'}=1$ it is seen that
\[\begin{split}
\bigl\|\|\cdot\|^{\alpha\cdot\frac{a}{2}}f\bigr\|^{1/\alpha}_{L^2_{k,a}} \|f\|^{1/\alpha'}_{L^2_{k,a}} & = \Bigl(\int_{\R^N}\|x\|^{\alpha a}\vert f(x)\vert^2\,d\mu_{k,a}(x)\Bigr)^{1/2\alpha} \Bigl(\int_{\R^N}\vert f(x)\vert^2\,d\mu_{k,a}(x)\Bigr)^{1/2\alpha'}\\ &=\bigl\|\|\cdot\|^{2\cdot\frac{a}{2}}\vert f\vert^{2/\alpha}\bigr\|^{1/2}_{L^\alpha_{k,a}} \bigl\|\vert f\vert^{2/\alpha'}\bigr\|^{1/2}_{L_{k,a}^{\alpha'}},
\end{split}\]
and furthermore by H\"older's inequality that 
\[\begin{split}
\bigl\|\|\cdot\|^{a/2}f\bigr\|^2_{L_{k,a}^2}&\leq\Bigl(\int_{\R^N}(\|x\|^a\vert f(x)\vert^{2/\alpha})^\alpha\,d\mu_{k,a}(x)\Bigr)^{1/\alpha} \Bigl(\int_{\R^N}(\vert f(x)\vert^{2/\alpha'})^{\alpha'}\,d\mu_{k,a}(x)\Bigr)^{1/\alpha'}\\ &=\Bigl(\int_{\R^N}\|x\|^{\alpha\cdot a}\vert f(x)\vert^2\,d\mu_{k,a}(x)\Bigr)^{1/\alpha}\Bigl(\int_{\R^N}\vert f(x)\vert^2\,d\mu_{k,a}(x)\Bigr)^{1/\alpha'},
\end{split}\]
from which we obtain the inequality $\bigl\|\|\cdot\|^{a/2}f\bigr\|_{L^2_{k,a}}\leq \bigl\|\|\cdot\|^{\alpha\cdot\frac{a}{2}}f\bigr\|^{1/\alpha}_{L^2_{k,a}} \|f\|_{L^2_{k,a}}^{1/\alpha'}$, that is
\begin{equation}\label{eqn.ineq.alpha}
\forall\alpha\geq 1: \bigl\|\|\cdot\|^{\alpha\cdot\frac{a}{2}}f\bigr\|^{1/\alpha}_{L^2_{k,a}} \geq\frac{\bigl\|\|\cdot\|^{a/2}f\bigr\|_{L^2_{k,a}}}{\|f\|^{1-\frac{1}{\alpha}}_{L^2_{k,a}}}.
\end{equation}
The same argument applied to $\mathcal{F}_{k,a}f$ leads to the analogous inequality
\begin{equation}\label{eqn.ineq.beta}
\forall\beta\geq 1: \bigl\|\|\cdot\|^{\beta\cdot\frac{a}{2}}\mathcal{F}_{k,a}f\bigr\|^{1/\beta}_{L^2_{k,a}} \geq \frac{\bigl\|\|\cdot\|^{a/2}\mathcal{F}_{k,a}f\bigr\|_{L^2_{k,a}}}{\|\mathcal{F}_{k,a}f\|_{L^2_{k,a}}^{1-\frac{1}{\beta}}}.
\end{equation}
We conclude from \eqref{eqn.ineq.alpha}, \eqref{eqn.ineq.beta}, and Theorem \ref{thm.Heisenberg-BSKO} that
\[\begin{split}
\bigl\|\|\cdot\|^{\alpha\cdot\frac{a}{2}}f\bigr\|_{L^2_{k,a}}^{\frac{\beta}{\alpha+\beta}} \bigl\|\|\cdot\|^{\beta\cdot\frac{a}{2}}\mathcal{F}_{k,a}f\bigr\|_{L^2_{k,a}}^{\frac{\alpha}{\alpha+\beta}} &= \Bigl[\bigl\|\|\cdot\|^{\alpha\cdot\frac{a}{2}}f\bigr\|_{L^2_{k,a}}^{1/\alpha} \bigl\|\|\cdot\|^{\beta\cdot\frac{a}{2}}\mathcal{F}_{k,a}f\bigr\|_{L^2_{k,a}}^{1/\beta}\Bigr]^{\frac{\alpha\beta}{\alpha+\beta}}  \\ &\geq \biggl(\frac{\bigl\|\|\cdot\|^{a/2}f\bigr\|_{L^2_{k,a}} \bigl\|\|\cdot\|^{a/2}\mathcal{F}_{k,a}f\bigr\|_{L^2_{k,a}}}{\|f\|_{L^2_{k,a}}^{1-\frac{1}{\alpha}}\|f\|_{L^2_{k,a}}^{1-\frac{1}{\beta}}}\biggr)^{\frac{\alpha\beta}{\alpha+\beta}}\\
&\geq \Bigl(\frac{2\left<k\right>+N+a-2}{2}\Bigr)^\frac{\alpha\beta}{\alpha+\beta}\|f\|_{L^2_{k,a}}^{(2-2+\frac{1}{\alpha}+\frac{1}{\beta})\frac{\alpha\beta}{\alpha+\beta}}
 \end{split}\]
 which is exactly the asserted inequality.
\end{proof}

\begin{remark}
Our theorem \ref{prop.HPW} is a slight improvement of \cite[Theorem~4.4(3)]{Ghobber-Jaming-studia} since we obtain a better constant. This is to be expected, however, since the point of \cite{Ghobber-Jaming-studia} is to obtain uncertainty principles for large classes of integral transforms. In concrete situations more detailed information can be brought to bear, as in theorem \ref{thm.Heisenberg-BSKO} where an optimal constant could be found. It is unlikely that comparably sharp results for general integral transforms can be established.
\end{remark}
As our proof relies on H\"older's inequality, it cannot include the cases $0<\alpha,\beta<1$.  As far as we could ascertain from the existing literature, in particular  \cite{Ciatti-Ricci-Sundari} and \cite{Martini}, most proofs of such an improvement involve heat kernel estimates either directly or disguised in spectral estimates of powers of the Laplacian. The heat kernel for the operator $-\Delta_{k,a}$ is only known at present in the cases (i) $N=1$, $a>0$ (where one can `deform' the known one-dimensional Dunkl-heat kernel with the parameter $a$), and (ii) $N\geq 2, a\in\{1,2\}$ (where the explicit formula was obtained in \cite{BS-semi}), and even in those cases it is nontrivial to obtain the required bounds. In this regard the techniques employed in \cite{Ghobber-Jaming-studia} are more suitable, since they merely require that the kernel of the integral transform be suitably bounded.

\begin{theorem} 
Assume $N\geq 1$, $k\geq 0$, that $a>0$ satisfies $a+2\left<k\right>+N>2$, and that $0<\alpha,\beta<1$. If either 
\begin{enumerate}[label=(\roman*)]
\item $N=1$ and $a>0$,
\item $N\geq 2$ and $a\in\{1,2\}$,
\end{enumerate}
or
\begin{enumerate}[resume,label=(\roman*)]
\item $N=2$, $k\equiv 0$, $a=2/n$ for some $n\in\N$,
\end{enumerate}
there exists a finite constant $c=c(\alpha,\beta)$ such that
\[\bigl\| \|\cdot\|^{\alpha\cdot\frac{a}{2}}f\bigr\|_{L^2_{k,a}}^{\frac{\beta}{\alpha+\beta}} \bigl\|\|\cdot\|^{\beta\cdot\frac{a}{2}}\mathcal{F}_{k,a}f\bigr\|_{L^2_{k,a}}^{\frac{\alpha}{\alpha+\beta}}\geq c(\alpha,\beta)\|f\|_{L^2_{k,a}}.\]
\end{theorem}
\begin{proof}
The statement follows from \cite[Theorem~C]{Ghobber-Jaming-studia} since the kernel $B_{k,a}$ for the $(k,a)$-generalized Fourier transform $\mathcal{F}_{k,a}$ is uniformly bounded in all the cases listed in the statement of the theorem. 
\end{proof}

The last variation on the theme of Heisenberg inequalities incorporates $L^p$-norms and is based on the following substitute for the heat kernel decay estimates that were used in \cite{Ciatti-Ricci-Sundari}. Recall that $\gamma_t(y)=\exp(-t\|y\|^2)$.

\begin{lemma}\label{lemma-tech}
Let $p\in(1,2]$, $q=p'=\frac{p}{p-1}$, and $0<\alpha<\frac{a(\nu_a+1)}{q}=\frac{2\left<k\right>+N-1}{q}$. For every $f\in L^p_{k,a}(\R^N)$ and $t>0$,
\[\bigl\| \gamma_t\mathcal{F}_{k,a}f\bigr\|_{L^q_{k,a}} \leq  \Bigl(1+\frac{K^{2/q}}{(a(\nu_a+1)-\alpha q)^{1/p}}\Bigl(\frac{\Gamma(\nu_a+1)}{2q^{\nu_a+1}}\Bigr)^{1/q}\Bigr)t^{-\alpha/a}\bigl\|\|\cdot\|^\alpha f\bigr\|_{L^p_{k,a}}.\]
\end{lemma}
The constant is not optimal; what is important is the exponent $-\alpha/a$ in the decay rate of $t$. Also note that we could have used the weight $\vert\cdot\vert^{\alpha\cdot\frac{a}{2}}$ on the right hand side, as in  proposition \ref{prop.HPW}. One would then have to impose the restriction $0<\alpha\frac{a}{2}<\frac{2\left<k\right>+N-1}{q}$ which translates into the condition $0<\alpha<\frac{2(2\left<k\right>+N-1)}{aq}$ which involves $a$. This would lead to an $a$-independent decay factor $t^{-\alpha/2}$ in place of $t^{-\alpha/a}$, so it is a matter of scaling.

\begin{corollary}\label{corollary-tech}
Let $p\in(1,2]$, $q=p'=\frac{p}{p-1}$, and $0<\alpha<\frac{2(2\left<k\right>+N-1)}{aq}$. For every $f\in L^p_{k,a}(\R^N)$ and $t>0$,
\[\bigl\| \gamma_t\mathcal{F}_{k,a}f\bigr\|_{L^q_{k,a}} \leq  \Bigl(1+\frac{K^{2/q}}{(a(\nu_a+1)-\alpha q\cdot\frac{a}{2})^{1/p}}\Bigl(\frac{\Gamma(\nu_a+1)}{2q^{\nu_a+1}}\Bigr)^{1/q}\Bigr)t^{-\alpha/2}\bigl\|\|\cdot\|^{\alpha\cdot\frac{a}{2}} f\bigr\|_{L^p_{k,a}}.\]
\end{corollary}
\begin{proof}[Proof of lemma \ref{lemma-tech}]
Assume without loss of generality that $\|\|\cdot\|^\alpha f\|_{L^q_{k,a}}$ is finite. Since $(\|x\|/r)^\alpha\geq 1$ whenever $x\in \complement B_r$, where $B_r=\{x\in\R^N\,:\,\|x\|\leq r\}$, it holds that $\vert(f\mathbf{1}_{\complement B_r})(x)\vert\leq \|x/r\|^\alpha\vert f(x)\vert$ for every $x\in\R^N$, from which it follows that
\[\begin{split}
\bigl\|\gamma_t\mathcal{F}_{k,a}(f\mathbf{1}_{\complement B_r})\bigr\|_{L^q_{k,a}} &\leq \|\gamma_t\|_{L^\infty_{k,a}} \bigl\|\mathcal{F}_{k,a}(f\mathbf{1}_{\complement B_r})\bigr\|_{L^q_{k,a}} \\
&\leq \|f\mathbf{1}_{\complement B_r}\|_{L^p_{k,a}}\quad\text{by proposition \ref{prop.HY}}\\
&\leq r^{-\alpha}\bigl\|\|\cdot\|^\alpha f\bigr\|_{L^p_{k,a}}.
\end{split}\]
In addition it holds by the H\"older inequality that 
\[\|\gamma_t\mathcal{F}_{k,a}(f\mathbf{1}_{B_r})\|_{L^q_{k,a}}\leq\|\gamma_t\|_{L^q_{k,a}}\|\mathcal{F}_{k,a}(f\mathbf{1}_{B_r})\|_{L^\infty_{k,a}}\leq \|\gamma_t\|_{L^\infty_{k,a}}\|f\mathbf{1}_{B_r}\|_{L^1_{k,a}}\leq 
\|\gamma_t\|_{L^q_{k,a}} \|g_\alpha\|_{L^q_{k,a}} \bigl\|\|\cdot\|^\alpha f\bigr\|_{L^p_{k,a}}\] The norms $\|\gamma_t\|_{L^q_{k,a}}$ and $\|g_\alpha\|_{L^q_{k,a}}$ having already been computed in \eqref{eqn.norm-gaussian} and \eqref{eqn.norm-f-alpha}, respectively, we conclude that
\[\|\gamma_t\mathcal{F}_{k,a}(f\mathbf{1}_{B_r})\|_{L^q_{k,a}} \leq 
\frac{K^{2/q}}{(a(\nu_a+1)-\alpha q)^{1/q}}\Bigl(\frac{\Gamma(\nu_a+1)}{2q^{\nu_a+1}}\Bigr)^{1/q}t^{-\frac{\nu_a+1}{q}}r^{\frac{a(\nu_a+1)}{q}-\alpha} \bigl\|\|\cdot\|^\alpha f\bigr\|_{L^p_{k,a}}\] and
\[\begin{split}
\|\gamma_t\mathcal{F}_{k,a}f\|_{L^q_{k,a}}&\leq \bigl\|\gamma_t\mathcal{F}_{k,a}(f\mathbf{1}_{B_r})\bigr\|_{L^q_{k,a}} + \bigl\|\gamma_t\mathcal{F}_{k,a}(f\mathbf{1}_{\complement B_r})\bigr\|_{L^q_{k,a}}\\ &
\leq \Bigl(1+\frac{K^{2/q}}{(a(\nu_a+1)-\alpha q)^{1/p}}\Bigl(\frac{\Gamma(\nu_a+1)}{2q^{\nu_a+1}}\Bigr)^{1/q}t^{-\frac{\nu_a+1}{q}}r^{\frac{a(\nu_a+1)}{q}}\Bigr)r^{-\alpha}\bigl\|\|\cdot\|^\alpha f\bigr\|_p.
\end{split}\]
This inequality holds, in particular, for $r=t^{1/a}$, from which the assertion follows.
\end{proof}

\begin{theorem} Under the same assumptions as in lemma \ref{lemma-tech} and with $\beta>0$, there exists a finite constant $c(\alpha,\beta)$ such that 
\[\|\mathcal{F}_{k,a}f\|_{L^q_{k,a}} \leq c(\alpha,\beta)\bigl\|\|t\cdot\|^\alpha f\bigr\|_{L^p_{k,a}}^{\frac{\beta}{\alpha+\beta}} \bigl\|\|\cdot\|^\beta\mathcal{F}_{k,a}f\bigr\|_{L^q_{k,a}}^{\frac{\alpha}{\alpha+\beta}}\text{ for all } f\in L^p_{k,a}(\R^N).\]
\end{theorem}
The proof that follows provides a rough estimate for $c(\alpha,\beta)$ but will be far from optimal. The main idea in the proof is to estimate the size of $\mathcal{F}_{k,a}f$ at two different scales. As such it follows closely the strategy in \cite{Ciatti-Ricci-Sundari}, although we replace their spectral estimates with estimates for $\mathcal{F}_{k,a}$.

\begin{proof} Fix $p\in(1,2]$ and assume $f\in L^p_{k,a}$ satisfy $\|\|\cdot\|^\alpha f\|_{L^p_{k,a}}+\|\|\cdot\|^\beta\mathcal{F}_{k,a}f\|_{L^q_{k,a}}<\infty$. Moreover assume that 
$\beta\leq a$. It follows for all $t>0$ from lemma \ref{lemma-tech} that
\begin{multline*}
\|\mathcal{F}_{k,a}f\|_{L^q_{k,a}} \leq \|\gamma_t\mathcal{F}_{k,a}f\|_{L^q_{k,a}} + \|(1-\gamma_t)\mathcal{F}_{k,a}f\|_{L^q_{k,a}}\\
\leq \Bigl(1+\frac{K^{2/q}}{(a(\nu_a+1)-\alpha q)^{1/p}}\Bigl(\frac{\Gamma(\nu_a+1)}{2q^{\nu_a+1}}\Bigr)^{1/q}\Bigr)t^{-\alpha/a}\bigl\|\vert\cdot\vert^\alpha f\bigr\|_{L^p_{k,a}} + \|(1-\gamma_t)\mathcal{F}_{k,a}f\bigr\|_{L^q_{k,a}}.
\end{multline*}
Moreover $\|(1-\gamma_t)\mathcal{F}_{k,a}f\|_{L^q_{k,a}}=t^{\beta/a}\|(t\|\cdot\|^a)^{-\beta/a}(1-\gamma_t)\|\cdot\|^\beta\mathcal{F}_{k,a}f\|_{L^q_{k,a}}$, where
\[\|(t\|\cdot\|^a)^{-\beta/a}(1-\gamma_t)\|\cdot\|^\beta\mathcal{F}_{k,a}f\|_{L^q_{k,a}}\leq \|(t\|\cdot\|^a)^{-\frac{\beta}{a}}(1-\gamma_t)\|_{L^\infty_{k,a}}\bigl\|\|\cdot\|^{\beta}\mathcal{F}_{k,a}f\bigr\|_{L^q_{k,a}}=c\bigl\|\|\cdot\|^\beta\mathcal{F}_{k,a}f\bigr\|_{L^q_{k,a}}\]
whenever $0<\beta\leq a$. It follows that $\|\mathcal{F}_{k,a}f\|_{L^q_{k,a}}\leq c(t^{-\alpha/a}\||t\cdot\|^\alpha f\|_{L^p_{k,a}}+t^{\beta/a}\|\|\cdot\|^\beta\mathcal{F}_{k,a}f\|_{L^q_{k,a}})$ for all $t>0$. The  choice $t=\bigl(\frac{\alpha}{\beta}\frac{\|\|\cdot\|^\alpha f\|_{L^p_{k,a}}}{\|\|\cdot\|^\beta\mathcal{F}_{k,a}f\|_{L^q_{k,a}}}\bigr)^{\frac{a}{\alpha+\beta}}$, in particular, gives rise to the inequality
\[\|\mathcal{F}_{k,a}f\|_{L^q_{k,a}} \leq c\Bigl(\bigl(\tfrac{\beta}{\alpha}\bigr)^{\frac{\alpha}{\alpha+\beta}}+\bigl(\tfrac{\alpha}{\beta}\bigr)^{\frac{\beta}{\alpha+\beta}}\Bigr)\bigl\|\|\cdot\|^\alpha f\bigr\|_{L^p_{k,a}}^{\frac{\beta}{\alpha+\beta}} \bigl\|\|\cdot\|^\beta\mathcal{F}_{k,a}f\bigr\|_{L^q_{k,a}}^{\frac{\alpha}{\alpha+\beta}}\text{ for }\beta\leq a.\]
\medskip

The remaining case $\beta>a$ can be treated by a slight variation of the arguments already given. Since $u^{a/2}\leq 1+u^\beta$ for all $u\geq 0$, it follows in particular for $u$ of the form $u=\|y\|/\epsilon$, $\epsilon$ an arbitrary positive parameter, that $(\|y\|/\epsilon)^{a/2}\leq 1+(\|y\|/\epsilon)^\beta$ for all $\epsilon>0$. There is nothing special about $a/2$, any exponent less than $a$ would work, since we may then apply the first part of the proof.

Therefore \[\|\|\cdot\|^{a/2}\mathcal{F}_{k,a}f\|_{L^q_{k,a}}\leq\epsilon^{a/2}\|\mathcal{F}_{k,a}f\|_{L^q_{k,a}}+\epsilon^{a/2-\beta}\|\|\cdot\|^\beta\mathcal{F}_{k,a}f\|_{L^q_{k,a}}\] for all $\epsilon>0$. In particular, by choosing $\epsilon$ such that $\epsilon^{a/2}\|\mathcal{F}_{k,a}f\|_{L^q_{k,a}}=\epsilon^{a/2-\beta}\|\|\cdot\|^\beta\mathcal{F}_{k,a}f\|_{L^q_{k,a}}$ (which amounts to taking $\epsilon=\|\|\cdot\|^\beta\mathcal{F}_{k,a}f\|^{1/\beta}\|\mathcal{F}_{k,a}f\|_{L^q_{k,a}}^{-1/\beta}$), it follows that
\[\bigl\|\|\cdot\|^{a/2}\mathcal{F}_{k,a}f\bigr\|_{L^q_{k,a}}\leq 2\|\mathcal{F}_{k,a}f\|_{L^q_{k,a}}^{\frac{2\beta-a}{2\beta}}\bigl\|\|\cdot\|^\beta\mathcal{F}_{k,a}f\bigr\|_{L^q_{k,a}}^{\frac{a}{2\beta}},\]
whence
\[\begin{split}
\|\mathcal{F}_{k,a}f\|_{L^q_{k,a}} &\leq c\bigl\|\|\cdot\|^\alpha f\bigr\|_{L^p_{k,a}}^{\frac{a/2}{\alpha+a/2}}\bigl\|\|\cdot\|^{a/2}\mathcal{F}_{k,a}\bigr\|_{L^q_{k,a}}^{\frac{\alpha}{\alpha+a/2}}\quad\text{by the first part of the proof}\\
&\leq c'\bigl\|\|\cdot\|^\alpha f\bigr\|_{L^p_{k,a}}^{\frac{a/2}{\alpha+a/2}}\bigl\|\|\cdot\|^\beta\mathcal{F}_{k,a}f\bigr\|_{L^q_{k,a}}^{\frac{a}{2\beta}\frac{\alpha}{\alpha+a/2}}\|\mathcal{F}_{k,a}f\|_{L^q_{k,a}}^{\frac{2\beta-a}{2\beta}\frac{\alpha}{\alpha+a/2}}
\end{split}\]
Elementary algebra now leads to the desired conclusion in the case $\beta>a$ as well: Isolating all factors with $\|\mathcal{F}_{k,a}f\|_{L^q_{k,a}}$ on the left hand side of the inequality yields the exponent $1-\frac{2\beta-a}{2\beta}\frac{\alpha}{\alpha+a/2} = \frac{a(\alpha+\beta)}{2\beta(\alpha+a/2)}$, and $\frac{a/2}{\alpha+a/2}\frac{2\beta(\alpha+a/2)}{a(\alpha+\beta)}=\cdots=\frac{\beta}{\alpha+\beta}$, for example.
\end{proof}
The following alternative formulation follows by scaling, just as in corollary \ref{corollary-tech}. Note that the exponents $\frac{\alpha}{\alpha+\beta}$ and $\frac{\beta}{\alpha+\beta}$ are invariant under rescaling $\alpha\to\alpha\frac{a}{2}$, $\beta\to\beta\frac{a}{2}$. Moreover $d(\alpha,\beta)=c(\alpha\frac{a}{2},\beta\frac{a}{2})$. 
\begin{corollary}
Under the same assumptions as in corollary \ref{corollary-tech} and with $\beta>0$, there exists a finite constant $d(\alpha,\beta)$ such that
\[\|\mathcal{F}_{k,a}f\|_{L^q_{k,a}}\leq d(\alpha,\beta)\bigl\|\|\cdot\|^{\alpha\cdot\frac{a}{2}}f\bigr\|^{\frac{\beta}{\alpha+\beta}}_{L^p_{k,a}} \cdot \bigl\|\|\cdot\|^{\beta\cdot\frac{a}{2}}\mathcal{F}_{k,a}f\bigr\|^{\frac{\alpha}{\alpha+\beta}}_{L^q_{k,a}}\text{ for all } f\in L^p_{k,a}(\R^N).\]
\end{corollary}
\begin{remark}
It is possible to generate an abundance of additional inequalities similar to the aforementioned ones. The interested reader will quickly be able to generalize the results in \cite[Section~2]{Cowling-Price}, for example, since these inequalities all arise as the result of simple scaling properties.
\end{remark}

\section{A variation of the HPW inequality with $L^1$-norms}
Another variation involves a mixed $L^1$,$L^2$ lower bound and was recently obtained by Ghobber \cite{Ghobber-L1} for the Dunkl transform. Its Euclidean counterpart seems to go back to \cite{Laeng-Morpurgo}, \cite{Morpurgo}, where the best constant is determined. The proof is elementary and - like in \cite[Section~3]{Ghobber-L1} -- based on the following two inequalities, the contents of which are somewhat obscure, unfortunately (the complicated exponents all arise as a consequence of scaling and homogeneity properties of the underlying measures).

\begin{lemma}[Nash-type inequality]
Let $s>0$ and assume $a>0$ is chosen in such a way that the Plancherel theorem for $\mathcal{F}_{k,a}$ is valid. Then
\[\|\mathcal{F}_{k,a}f\|^2_{L^2_{k,a}}=\|f\|_{L^2_{k,a}}^2 \leq C\|f\|_{L^1_{k,a}}^{\frac{2s}{\left<k\right>+\frac{a+N}{2}+s-1}}\bigl\|\|\cdot\|^s\mathcal{F}_{k,a}f\bigr\|_{L^2_{k,a}}^{\frac{2\left<k\right>+a+N-2}{\left<k\right>+\frac{a+N}{2}+s-1}}\]
for every $f\in (L^1_{k,a}\cap L^2_{k,a})(\R^N)$, where
\[C=C(k,a,s)=\frac{K}{2\left<k\right>+a+N-2}\Bigl(\frac{2s}{K}\Bigr)^{\frac{2\left<k\right>+a+N-2}{2\left<k\right>+a+N+2s-2}}+\Bigl(\frac{2s}{K}\Bigr)^{-\frac{2s}{2\left<k\right>+a+N+2s-2}}.\]
\end{lemma}
\begin{proof}
For $f\in L^2_{k,a}(\R^N)$ and $r>0$ fixed, consider the function $\mathbf{1}_r=\mathbf{1}_{B_r(0)}$. It follows from the Plancherel theorem for $\mathcal{F}_{k,a}$ and the fact $\mathbf{1}_r(1-\mathbf{1}_r)\equiv 0$ that $\|f\|_{L^2_{k,a}}^2=\|\mathcal{F}_{k,a}f\|^2_{L^2_{k,a}}=\|(\mathcal{F}_{k,a}f)\mathbf{1}_r\|_{L^2_{k,a}}^2+\|(\mathcal{F}_{k,a}f)(1-\mathbf{1}_r)\|^2_{L^2_{k,a}}$, where
\[\begin{split}
\|(\mathcal{F}_{k,a}f)\mathbf{1}_r\|^2_{L^2_{k,a}} &=\int_{B_r(0)}\vert\mathcal{F}_{k,a}f(\xi)\vert^2\,d\mu_{k,a}(\xi)\leq \|\mathcal{F}_{k,a}f\|^2_{L^\infty_{k,a}}\mu_{k,a}(B_r(0)) \\ & \leq \frac{K}{2\left<k\right>+a+N-2}r^{2\left<k\right>+a+N-2} \|\mathcal{F}_{k,a}f\|^2_{L^\infty_{k,a}} \end{split}\]
and
\[\begin{split}
\|(\mathcal{F}_{k,a}f)(1-\mathbf{1}_r)\|^2_{L^2_{k,a}}& =\int_{\R^N\setminus B_r(0)}\vert\mathcal{F}_{k,a}f(\xi)\vert^2\,d\mu_{k,a}(\xi)\\ &\leq r^{-2s}\int_{\R^N\setminus B_r(0)}\|\xi\|^{2s}\vert\mathcal{F}_{k,a}f(\xi)\vert^2\,d\mu_{k,a}(\xi) = r^{-2s}\bigl\|\|\cdot\|^2\mathcal{F}_{k,a}f\bigr\|^2_{L^2_{k,a}}.\end{split}\]
Therefore
\[\|\mathcal{F}_{k,a}f\|^2_{L^2_{k,a}}\leq\frac{K}{2\left<k\right>+a+N-2}r^{2\left<k\right>+a+N-2}\|f\|^2_{L^1_{k,a}} + r^{-2s}\bigl\|\|\cdot\|^s\mathcal{F}_{k,a}f\bigr\|^2_{L^2_{k,a}},\]
the right hand side of which is minimized when $\displaystyle r^{2\left<k\right>+a+N+2s-2}=\frac{2s}{K}\frac{\|\|\cdot\|^s\mathcal{F}_{k,a}f\|^2_{L^2_{k,a}}}{\|f\|^2_{L^1_{k,a}}}$.
\end{proof}
\begin{lemma}[Clarkson-type inequality for $\vartheta_{k,a}(x)dx$]
Let $s>0$ and assume $a>0$ is chosen in such a way that the Plancherel theorem for $\mathcal{F}_{k,a}$ is valid. Then
\[\|f\|_{L^1_{k,a}} \leq D(k,a,s)\|f\|_{L^2_{k,a}}^{\frac{2s}{\left<k\right>+\frac{a+N}{2}+1+2s}} \bigl\|\|\cdot\|^{2s}f\bigr\|_{L^1_{k,a}}^{\frac{\left<k\right>+\frac{a+N}{2}-1}{\left<k\right>+\frac{a+N}{2}+1+2s}}\]
for every $f\in (L^1_{k,a}\cap L^2_{k,a})(\R^N)$ where the constant $D$ is computable yet far from optimal.
\end{lemma}
\begin{proof}
Let $f\in (L^1_{k,a}\cap L^2_{k,a})(\R^N)$ and consider $\mathbf{1}_r=\mathbf{1}_{B_r(0)}$, $r>0$. Since $\|f\|_{L^1_{k,a}}\leq\|f\mathbf{1}_r\|_{L^1_{k,a}}+\|f(1-\mathbf{1}_r)\|_{L^1_{k,a}}\leq \|f\|_{L^2_{k,a}}\|\mathbf{1}_r\|_{L^2_{k,a}}+r^{-2s}\|\|\cdot\|^{2s}f\|_{L^1_{k,a}}$, it follows that
\[\|f\|_{L^1_{k,a}}\leq \Bigl(\frac{K}{2\left<k\right>+a+N-2}\Bigr)^{1/2}r^{\left<k\right>+\frac{a+N}{2}-1}\|f\|_{L^2_{k,a}}+r^{-2s}\bigl\|\vert\cdot\vert^{2s}f\bigr\|_{L^1_{k,a}}\]
the right hand side of which is minimized for 
\[ r^{\left<k\right>+\frac{a+N}{2}+1+2s}=\frac{2s}{(\left<k\right>+\frac{a+N}{2}-1)\bigl(\frac{K}{2\left<k\right>+a+N-2}\bigr)^{1/2}}\|\|\cdot\|^{2s}f\|_{L^1_{k,a}}\|f\|_{L^2_{k,a}}^{-1}.\]
\end{proof}
The following uncertainty-type inequality follows at once by combining the aforementioned two lemmata, which at the same time yields an expression for the constant $C'$.
\begin{prop}
Let $s>0$ and assume $a>0$ is chosen in such a way that the Plancherel theorem for $\mathcal{F}_{k,a}$ is valid. Then there exists a constant $C'>0$ such that for all $f\in (L^1_{k,a}\cap L^2_{k,a})(\R^N)$
\[
\bigl\|\|\cdot\|^{2s}f\bigr\|_{L^1_{k,a}} \bigl\|\|\cdot\|^s\mathcal{F}_{k,a}f\bigr\|_{L^2_{k,a}}^2 \geq C'\|f\|_{L_{k,a}^1}\|f\|_{L^2_{k,a}}^2.
\]
\end{prop}

\section{Inequalities for Shannon entropy}\label{section.entropy}
It is the purpose of the present section to establish an analogue of Hirschman's entropic inequality for the $(k,a)$-generalized transform $\mathcal{F}_{k,a}$ and use it to give a new proof of the Heisenberg--Pauli--Weyl inequality.

\begin{theorem}\label{thm.Shannon}
Assume $N$, $k$ and $a$ satisfy either (i) or (ii) in lemma \ref{lemma.B-parameters}. 
For every $f\in L^2_{k,a}(\R^N)$ with $\|f\|_{L^2_{k,a}}=1$ it holds that
\[\entr{\vert f\vert^2}+\entr{\vert\mathcal{F}_{k,a}f\vert^2}\geq 0,\]
where
\[\entr{h}=-\int_{\R^N}\ln(|h(x)| |h(x)|\,d\mu_{k,a}(x).\]
\end{theorem}

In the case of Euclidean Fourier analysis the idea of proof is to differentiate the Hausdorff--Young inequality with respect to $p$, use various properties of the Fourier transform to establish the statement for $f\in L^1\cap L^2$ and finish the proof with an approximate identity-argument. This was worked out in some detail by Hirschman \cite{Hirschman-entropy} but might have been used even earlier. It has since become a standard tool in the field of geometric inequalities, be it Sobolev or Hardy--Littlewood--Sobolev inequalities in various settings. The idea is elementary and based on the following
\begin{lemma}\label{lemma.calculus}
Let $I=[1,2]$ and $\psi,\phi$ be real-valued differentiable functions on $I$ such that $\phi(t)\leq\psi(t)$ for $t\in I$ and $\phi(2)=\psi(2)$. Then $\phi'(2^-)\geq\psi'(2^-)$ (one-sided derivatives at $p=2$).
\end{lemma}

Note, however, that the lack of convolution structure necessitates a different kind of approximation argument. We shall use the Schwartz space $\mathscr{S}(\R^N)$ instead, in which case its invariance under $\mathcal{F}_{k,a}$, cf. lemma \ref{lemma.schwartz}, becomes important.

\begin{proof}[Proof of Theorem \ref{thm.Shannon}]
First assume that either one of the integrals
\[\int_{\R^N}|f(x)|^2\log^+|f(x)|^2\,d\mu_{k,a}(x),\quad\int_{\R^N}|f(x)|^2\log^-|f(x)|^2\,d\mu_{k,a}(x)\]
is finite. The quantity $\entr{|f|^2}+\entr{|\mathcal{F}_{k,a}f|^2}$ is therefore well-defined except when either
\begin{enumerate}[label=(\alph*)]
\item $\entr{|f|^2}$ or $\entr{|\mathcal{F}_{k,a}f|^2}$ is not defined 
\end{enumerate}
or
\begin{enumerate}[resume,label=(\alph*)]
\item $\entr{|f|^2}=\pm\infty$ and $\entr{|\mathcal{F}_{k,a}f|^2}=\mp\infty$. (It suffices to exclude the case $\entr{|f|^2}=+\infty$, $\entr{|\mathcal{F}_{k,a}f|^2}=-\infty$).
\end{enumerate}

Let $f\in\mathscr{S}(\R^N)$ and $p\in[1,2]$ be fixed and define $r(p)=\frac{\|\mathcal{F}_{k,a}f\|_{k,p'}}{\|f\|_{k,p}}$ together with
\[C(p)=\log r(p) = \frac{1}{p'}\log\Bigl(\int_{\R^N}\vert\mathcal{F}_{k,a}f(\xi)\vert^{p'}d\mu_{k,a}(\xi)\Bigr)-\frac{1}{p}\log\Bigl(\int_{\R^N}\vert f(x)\vert^pd\mu_{k,a}(x)\Bigr).\]
Then $C(2)=0$ and $C(p)\leq 0$ for $1<p<2$, by the Hausdorff--Young inequality, and the one-sided derivative $C'(2^-)$ -- whenever it exists -- will be seen to be strictly positive. Let
\[\psi:\R^N\times [1,2]\to\R,\quad (x,p)\mapsto\frac{\vert f(x)\vert^2-\vert f(x)\vert^p}{2-p}.\]
The functions $\psi_p:\R^N\to\R, x\mapsto \psi(x,p)$, $p\in(1,2]$, are seen to be integrable with respect to $\vartheta_{k,a}(x)dx$, and ${_x}\psi:(1,2]\to \R, p\mapsto\psi(x,p)$, converges towards $\vert f(x)\vert^2\log\vert f(x)\vert$ as $p\nearrow 2$. Define
\[
A(p)=\int_{\R^N}\vert f(x)\vert^pd\mu_{k,a}(x)\quad\text{and}\quad
B(q)=\int_{\R^N}\vert\mathcal{F}_{k,a}f(\xi)\vert^qd\mu_{k,a}(\xi).
\]
Then 
\[
\frac{A(2-h)-A(2)}{h}=\int_{\R^N}\frac{\vert f(x)\vert^{2-h}-\vert f(x)\vert^2}{h}\,d\mu_{k,a}(x)
\longrightarrow\int_{\R^N}\vert f(x)\vert^2\log\vert f(x)\vert\,d\mu_{k,a}(x)\]  
as $h\to 0, h>0$, that is, $A'(2^-)=\int\vert f(x)\vert^2\log\vert f(x)\vert\,d\mu_{k,a}(x)$. An analogous consideration shows that $B'(2^+)=\int \vert\mathcal{F}_{k,a}f(\xi)\vert^2\log\vert\mathcal{F}_{k,a}f(\xi)\vert\,d\mu_{k,a}(\xi)$,  and it follows that $C'(2^-)=-\frac{1}{2}B'(2^+)-\frac{1}{2}A'(2^-)$. Indeed
\[\begin{split}
\frac{C(p)-C(2)}{2-p} &= \frac{\frac{1}{p'}\log B(p')-\frac{1}{p}\log A(p)}{2-p} - \frac{\frac{1}{2}\log B(2)-\frac{1}{2}\log A(2)}{2-p}\\
&\longrightarrow -\frac{1}{2}\frac{B'(2^+)}{B(2)} - \frac{1}{2}\frac{A'(2^-)}{A(2)}\quad\text{ as } p\to 2, 1<p<2
\end{split}\]
where it was used that $B(2)=A(2)$ (by the unitarity of $\mathcal{F}_{k,a}$). Since  $A(2)=1$ by assumption, it even follows that $C'(2^-)=-\frac{1}{2}B'(2^+)-\frac{1}{2}A'(2^-)$ as claimed. In other words, $C'(2^-)=\entr{\vert f\vert^2}+\entr{\vert\mathcal{F}_{k,a}f\vert^2}$, and it remains to establish that $C'(2^-)>0$.  This follows from the elementary Lemma \ref{lemma.calculus}: Since $r(p)\leq 1$ for $1\leq p\leq 2$, with equality at $p=2$, we apply the lemma to the function $p\mapsto \log(r(p))$ to conclude that
\begin{equation}\label{eqn.Cprime}
C'(2^-)=\frac{r'(2^-)}{r(2)} \geq 0\end{equation}
which yields the asserted entropy inequality under the stronger assumption that $f\in\mathscr{S}(\R^N)$. Since the sharp Hausdorff--Young inequality is not presently known, it is very likely that the lower bound in \eqref{eqn.Cprime} can be improved considerably. We have tacitly excluded the case where $\entr{\vert f\vert^2}=+\infty$ and $\entr{\vert\mathcal{F}_{k,a}f\vert^2}=-\infty$. If we drop the requirement that $\|f\|_{L^2_{k,a}}=1$, the resulting entropic inequality becomes
\begin{equation}\label{ineq.general}
\frac{\entr{|f|^2}}{\|f\|^2_{L^2_{k,a}}} + \frac{\entr{|\mathcal{F}_{k,a}f|^2}}{\|\mathcal{F}_{k,a}f\|^2_{L^2_{k,a}}} \geq \log\|f\|^2_{L^2_{k,a}}+\log\|\mathcal{F}_{k,a}f\|^2_{L^2_{k,a}}.\end{equation}
\medskip

Now assume that $f$ is merely in $L^2_{k,a}(\R^N)$ with $\|f\|_2=1$, and choose a sequence $\{f_n\}$ in $\mathscr{S}(\R^N)$ such that $\lim_n\|f-f_n\|_2=0$. It follows from \eqref{ineq.general} that 
\[\frac{\entr{|f_n|^2}}{\|f_n\|_{L^2_{k,a}}} + \frac{\entr{|\mathcal{F}_{k,a}f_n|^2}}{\|\mathcal{F}_{k,a}f_n\|_{L^2_{k,a}}}\geq \log \|f_n\|_{L^2_{k,a}}^2+\log\|\mathcal{F}_{k,a}f_n\|^2_{L^2_{k,a}}\]
for all $n\in\N$, and by Lebesgue's theorem on majorized convergence that $\lim_n\entr{|f_n|^2}=\entr{|f|^2}$.  Since $\|f-f_n\|_2=\|\mathcal{F}_{k,a}(f-f_n)\|_2=\|\mathcal{F}_{k,a}f-\mathcal{F}_{k,a}f_n\|_2$, where $\mathcal{F}_{k,a}f_n$ is a Schwartz function according to lemma \ref{lemma.schwartz}, it follows that $\lim_n\|\mathcal{F}_{k,a}f-\mathcal{F}_{k,a}f_n\|_2=0$ and by Lebesgue that $\lim_n\entr{|\mathcal{F}_{k,a}f_n|^2}=\entr{|\mathcal{F}_{k,a}f|^2}$. We conclude that $\entr{|f|^2}+\entr{|\mathcal{F}_{k,a}f|^2}\geq 0$.
\end{proof}
Although the entropic inequality for $\mathcal{F}_{k,a}$ on $\R^N$ is not sharp, it still yields further inequalities. One can establish the Heisenberg--Pauli--Weyl inequality, for example, in a form that improves proposition \ref{prop.HPW} by allowing more freedom in the choice of power weights. This type of argument was also used in the unpublished preprint \cite{Dhaouadi}.

Let $\alpha,c$ be fixed, positive numbers and define constants
\[\sigma_\alpha=\int_{\R^N}e^{-\|x\|^\alpha}\vartheta_{k,a}(x)\,dx,\quad k_{\alpha,c}=\frac{\sigma_\alpha}{c^{2\langle k\rangle+a-N-2}}.\]
Then $d\gamma(x)=k_{\alpha,c}^{-1}\exp(-\|cx\|^\alpha)\vartheta_{k,a}(x)\,dx$ defines a probability measure on $\R^N$, since
\[\begin{split}
\int_{\R^N}d\gamma(x)&=\frac{1}{k_{\alpha,c}}\int_{\R^N}e^{-\|cx\|^\alpha}\vartheta_{k,a}(x)\,dx = \frac{1}{k_{\alpha,c}}c^{N-(2\langle k\rangle+a-2)}\int_{\R^N}e^{-\|x\|^\alpha}\vartheta_{k,a}(x)\,dx \\
&=\frac{c^{2\langle k\rangle+a-N-2}}{\sigma_\alpha} c^{N-2\langle k\rangle-a+2}\sigma_\alpha=1
\end{split}\]
Let $\phi\in L^1_{k,a}(\R^N)$ with $\|\phi\|_{L^1_{k,a}}=1$ be fixed and consider the function defined by $\psi(x)=k_{\alpha,c}\exp(\|cx\|^\alpha)|\phi(x)|$. Then $\|\psi\|_{L^1(\gamma)}=1$, and it follows from Jensen's inequality applied to the convex function $g:[0,\infty)\to\R:~t\mapsto t\ln t$ that
\[\begin{split}
0&=g\Bigl(\int_{\R^N}\psi(x)\,d\gamma(x)\Bigr) =\Bigl(\int_{\R^N}\psi(x)\,d\gamma(x)\Bigr)\ln\Bigl(\int_{\R^N}\psi(x)\,d\gamma(x)\Bigr)\\
&\leq \int_{\R^N}\psi(x)\ln(\psi(x))\,d\gamma(x) \leq \int_{\R^N}|\phi(x)|\bigl(\ln k_{\alpha,x}+\|cx\|^\alpha+\ln(|\phi(x)|)\bigr)\vartheta_{k,a}(x)\,dx\\
&=\ln k_{\alpha,c}+c^\alpha\int_{\R^N}\|x\|^\alpha|\phi(x)|\,d\mu_{k,a}(x) - \entr{|\phi|}
\end{split}\]
that is,
\begin{equation}\label{eqn.entropy-variance}
\entr{|\phi|}\leq \ln k_{\alpha,c}+c^\alpha(M_\alpha(\phi))^\alpha,
\end{equation}
where
\[M_\alpha(\phi)=\int_{\R^N}\|x\|^\alpha|\phi(x)|\,d\mu_{k,a}(x)\]
is a generalized variance of the probability density $\phi$. In particular \eqref{eqn.entropy-variance} holds for $\rho=|f|^2$ resp. $\rho=|\mathcal{F}_{k,a}f|^2$, where $f\in L^2_{k,a}(\R^N)$ with $\|f\|_{L^2_{k,a}}=1$, that is,
\[\entr{|f|^2} \leq \ln k_{\alpha,c}+c^\alpha\int_{\R^N}\|x\|^\alpha|f(x)|^2\,d\mu_{k,a}(x)
=\ln k_{\alpha,c}+c^\alpha\bigl\| \|\cdot\|^{\alpha/2}f\bigr\|^2_{L^2_{k,a}}\]
and $\entr{|\mathcal{F}_{k,a}f|^2} \leq \ln k_{\beta,d}+d^\beta\bigl\| \|\cdot\|^{\beta/2}\mathcal{F}_{k,a}f\bigr\|^2_{L^2_{k,a}}$ for further constants $\beta,d>0$. It follows from theorem \ref{thm.Shannon} that 
\[
0 \leq\entr{|f|^2}+\entr{|\mathcal{F}_{k,a}f|^2}
\leq\ln(k_{\alpha,c}k_{\beta,d})+c^\alpha\bigl\| \|\cdot\|^{\alpha/2}f\bigr\|^2_{L^2_{k,a}} + d^\beta\bigl\|\|\cdot\|^{\beta/2}\mathcal{F}_{k,a}f\bigr\|^2_{L^2_{k,a}}.
\]
For more general $f\in L^2_{k,a}(\R^N)$, $f\neq 0$, we replace $f$ by $f/\|f\|_2$ to obtain the inequality
\begin{equation}\label{ineq.entropy2}
-\ln(k_{\alpha,c}k_{\beta,d})\|f\|^2_{L^2_{k,a}} \leq c^\alpha\bigl\|\|\cdot\|^{\alpha/2}f\bigr\|^2_{L^2_{k,a}} + d^\beta\bigl\|\|\cdot\|^{\beta/2}\mathcal{F}_{k,a}f\bigr\|^2_{L^2_{k,a}}.
\end{equation}
\begin{corollary}
There exists a constant $K=K_\alpha>0$ such that 
\[\bigl\|\|\cdot\|^{\alpha/2}f\bigr\|_{L^2_{k,a}}\cdot\bigl\|\|\cdot\|^{\alpha/2}\mathcal{F}_{k,a}f\bigr\|_{L^2_{k,a}}\geq K\|f\|^2_{L^2_{k,a}}\text{ for all  } f\in L^2_{k,a}(\R^N).\]
\end{corollary}
The constant $K$ can be computed by working through a scaling/dilation argument similar to the one following remark \ref{remark-scaling} above. Specifically, one chooses $\alpha=\beta$ and $c=d$ in the preceding considerations leading up to \eqref{ineq.entropy2}. Then replace $f$ by its dilation $f_t(x)=f(tx)$ and optimize in the variable $t$ to obtain the stated inequality. This provides an alternative proof of the Heisenberg--Pauli--Weyl uncertainty inequality by Ben Sa{\"\i}d, Kobayashi, and \O rsted, albeit \emph{without} recovering the optimal constant.

\begin{remark}
In recent years several generalizations of the Shannon entropy and its implications for uncertainty of quantum measurements have appeared in the physics literature, most notably the Renyi entropy and related quantities in information theory, such as the Fisher information. It would take us too far afield to discuss these at any length but the interested reader may consult \cite{Physics}.
\end{remark}
\section{Weighted inequalities}\label{section-weighted}
The Hausdorff--Young inequality was but an elementary outcome of applying interpolation techniques to the transform $\mathcal{F}_{k,a}$. It is indeed possible to obtain more general \emph{weighted} inequalities, and the present section addresses these matters. For our purposes  would suffice to consider power weights, but it might be of independent interest to work for more general classes of weights. We shall be interesting in a weighted extension of the Hausdorff--Young inequality and an analogue of Pitt's inequality.

We remind the reader that the classical Pitt's inequality can be phrased as follows.

\begin{theorem}[Pitt's inequality]
Let $1<p\leq q<\infty$, choose $0<b<1/p^\prime$, set $\beta=1-\frac{1}{p}-\frac{1}{q}-b<0$, and define $v(x)=\vert x\vert^{bp}$ for $x\in\R$. There exists a constant $C>0$ such that
\[\Bigl(\int_{\widehat{\R}}\vert\widehat{f}(\xi)\vert^q\vert\xi\vert^{\beta q}\,d\xi\Bigr)^{1/q}\leq C\Bigl(\int_\R\vert f(x)\vert^p\vert x\vert^{bp}\,dx\Bigr)^{1/p}\]
for all $f\in L^p_v(\R)$. In particular $\widehat{f}$ is well-defined in this case.
\end{theorem}
Here $L^p_v(\R)$ denotes the space of equivalence classes of measurable functions $f$ on $\R$ for which $\int_\R\vert f(x)\vert^pv(x)\,dx<\infty$, and our initial interest in Pitt's inequality stems from its prominent role in work by Beckner, most notably \cite{Beckner-Pitt1} and later publications. In particular, Beckner determined the optimal constant in the important special case $p=q=2$, $b+\beta=0$: For $f\in\mathscr{S}(\R^N)$ and $0\leq b<N$,
\[\int_{\R^N}\|\xi\|^{-b}\vert\widehat{f}(\xi)\vert^2\,d\xi\leq C(b)\int_{\R^N}\|x\|^b\vert f(x)\vert^2\,dx,\quad\text{ where }\quad C(b)=\pi^b\biggl(\frac{\Gamma(\frac{N-b}{4})}{\Gamma(\frac{N+b}{4})}\biggr)^2\]
In particular $C(0)=1$, and the inequality is even an equality, according to the Plancherel theorem.
\medskip

As far as we know, an analogue of Pitt's inequality for $\mathcal{F}_{k,a}$ -- even without sharp constants -- is unknown for $N\geq 2$, $k\not\equiv 0$. As already mentioned in the introduction the secondary goal of our paper is to fill this gap. The impetus was provided by the intriguing paper \cite{Benedetto-Heinig} where Benedetto and Heinig used interpolation techniques and classical inequalities for rearrangements to establish the following very general weighted inequality for the Euclidean Fourier transform (although the constants that appear are not optimal, it will be important to have some control over them).  In order to explain the methodology we must introduce some more terminology.  Let $(X,\mu)$ be a measure space, where we assume for simplicity that $X\subset\R^N$, and let $f:X\to\C$ be $\mu$-measurable. The \emph{distribution function} $D_f:[0,\infty)\to[0,\infty)$ of $f$ is defined by $D_f(s)=\mu(\{x\in X\,:\,\vert f(x)\vert>s\}$. Two functions $f$ and $g$ on measure spaces $(X,\mu)$ and $(Y,\nu)$, respectively, are \emph{equimeasurable} if $D_f$ and $D_g$ coincide as functions on $[0,\infty)$. The \emph{decreasing rearrangement} of $f$ defined on $(X,\mu)$ is the function $f^*:[0,\infty)\to[0,\infty)$ defined by $f^*(t)=\inf\{s\geq 0\,:\, D_f(s)\leq t\}$. By convention $\inf\emptyset=\infty$, so that $f^*(t)=\infty$ whenever $D_f(s)>t$ for all $s\in[0,\infty)$.

For a given $\mu$-measurable function $f$ on $X$, $f^*$ is non-negative, decreasing and right continuous on $[0,\infty)$. Moreover $f$ and $f^*$ are equimeasurable when $f^*$ is considered as a Lebesgue measurable function on $[0,\infty)$, and for every $p\in(0,\infty)$ it holds that
\[\int_X\vert f(x)\vert^p\,d\mu(x)=p\int_0^\infty s^{p-1}D_f(s)\,ds = \int_0^\infty (f^*(t))^p\,dt,\]
cf. proposition 1.8 on page 43 in \cite{Bennett-Sharpley}.

They first establish the following result, which can be traced to old results by Jodeit and Torchinsky (we shall supply more detail in the appendix):

\begin{theorem}[Theorem B in \cite{Benedetto-Heinig}]\label{thm.B} Let $q\geq 2$. There is $K_q>0$ such that, for all $f\in L^1+L^2$ and for all $s\geq 0$, the inequality
\[\int_0^s(\widehat{f})^*(t)^q\,dt\leq K_q^q\int_0^s\Bigl(\int_0^{1/t}f^*(r)\,dr\Bigr)^q dt\]
holds.
\end{theorem}

\begin{theorem}\label{thm.1}
Let $u$ and $v$ be weight functions on $\R^N$, suppose $1<p,q<\infty$, and let $K$ be the constant from theorem \ref{thm.B} associated with the relevant index $\geq 2$. There is a positive constant $C$ such that, for all $f\in L^p_v(\R^N,dx)$, the inequality 
\begin{equation}\label{eqn.weighted}
\Bigl(\int_{\R^N}\vert\widehat{f}(\gamma)\vert^q u(\gamma)\,d\gamma\Bigr)^{1/q}\leq KC 
\Bigl(\int_{\R^N}\vert f(x)\vert^pv(x)\,dx\Bigr)^{1/p}
\end{equation}
holds in the following ranges and with the following constraints on $u$ and $v$:
\begin{enumerate}[label=(\roman*)]
\item $1<p\leq q<\infty$ and 
\[\sup_{s>0}\Bigl(\int_0^{1/s}u^*(t)\,dt\Bigr)^{1/q}\Bigl(\int_0^s\bigl(\bigl(\tfrac{1}{v}\bigr)^*\bigr)(t)^{p^\prime-1}\,dt\Bigr)^{1/p'}\equiv B_1<\infty;\]
\item for $1<q<p<\infty$ and 
\[\Bigl(\int_0^\infty\Bigl(\int_0^{1/s}u^*\Bigr)^{r/q}\Bigl(\int_0^s\bigl(\tfrac{1}{v}\bigr)^{*(p'-1)}\Bigr)^{r/q'}\bigl(\bigl(\tfrac{1}{v}\bigr)^*\bigr)(s)^{p'-1}ds\Bigr)^{2/r}\equiv B_2<\infty,\]
where $\frac{1}{r}=\frac{1}{q}-\frac{1}{p}$.
\end{enumerate}
The \textbf{best} constant $C$ in  \eqref{eqn.weighted} satisfies
\[C\leq B_1\begin{cases} (q')^{1/p'}q^{1/q}&\text{if } 1<p\leq q, q\geq 2\\
p^{1/q}(p^\prime)^{1/p'}&\text{if } 1<p\leq q<2\end{cases}\]
and $C\leq B_2q^{1/q}(p^\prime)^{1/q'}$ if $1<q<p<\infty$.
\end{theorem}

In the case of the Euclidean Fourier transform on $\R^N$, the Pitt inequality is obtained by choosing the weights $u(\xi)=\|\xi\|^\alpha$, $v(x)=\|x\|^l$, $\alpha<0$, $l>0$. Here $u^*(t)=c_\alpha t^{\alpha/N}$ and $(1/v)^*(t)=c_lt^{-l/N}$ for all $t>0$, where $c_\alpha$ and $c_l$ are suitable constants. The weight conditions in the aforementioned theorem are thereby valid if and only if $-N<\alpha$, $l<N(p-1)$, and 
\[\frac{1}{N}\Bigl(\frac{l}{p}+\frac{\alpha}{q}\Bigr)=\frac{1}{p'}-\frac{1}{q}=1-\frac{1}{p}-\frac{1}{q}.\]
The \emph{disadvantage} of employing such rearrangement and interpolation methods is that one generally picks up sub-optimal constants. In the special case where $u\equiv 1\equiv v$, one does not obtain the Plancherel theorem as a limiting case. We shall provide the details for $\mathcal{F}_{k,a}$ later in this section.

We outline in an appendix the minor modification required to establish the following analogue of theorem \ref{thm.1} for $\mathcal{F}_{k,a}$, For all results on $\mathcal{F}_{k,a}$ that follow it is to be understood that $k\geq 0$ and $a>0$ satisfies $a+2\left<k\right>+N>2$, and either
\begin{enumerate}[label=(\roman*)]
\item $N=1$ and $a>0$,
\item $N\geq 2$ and $a\in\{1,2\}$
\end{enumerate}
or
\begin{enumerate}[resume,label=(\roman*)]
\item $N=2$ and $a=2/n$ for some $n\in\N$.
\end{enumerate}

\begin{theorem}\label{thm.weighted}
Let $u$ and $v$ be weight functions on $\R^N$, suppose $1<p,q<\infty$, and let $K$ be the constant from theorem \ref{thm.B} associated with the relevant index $\geq 2$. There is a positive constant $C$ such that, for all $f\in L^p_v(\R^N,\vartheta_{k,a}(x)dx)$, the inequality 
\begin{equation}\label{eqn.weighted}
\Bigl(\int_{\R^N}\vert\mathcal{F}_{k,a}f(\xi)\vert^q u(\xi)\,d\mu_{k,a}(\xi)\Bigr)^{1/q}\leq KC\Bigl(\int_{\R^N}\vert f(x)\vert^pv(x)\,d\mu_{k,a}(x)\Bigr)^{1/p}
\end{equation}
holds in the following ranges and with the following constraints on $u$ and $v$:
\begin{enumerate}[label=(\roman*)]
\item $1<p\leq q<\infty$ and 
\[\sup_{s>0}\Bigl(\int_0^{1/s}u^*(t)\,dt\Bigr)^{1/q}\Bigl(\int_0^s\bigl(\bigl(\tfrac{1}{v}\bigr)^*\bigr)(t)^{p^\prime-1}\,dt\Bigr)^{1/p'}\equiv B_1<\infty;\]
\item for $1<q<p<\infty$ and 
\[\Bigl(\int_0^\infty\Bigl(\int_0^{1/s}u^*\Bigr)^{r/q}\Bigl(\int_0^s\bigl(\tfrac{1}{v}\bigr)^{*(p'-1)}\Bigr)^{r/q'}\bigl(\bigl(\tfrac{1}{v}\bigr)^*\bigr)(s)^{p'-1}ds\Bigr)^{2/r}\equiv B_2<\infty,\]
where $\frac{1}{r}=\frac{1}{q}-\frac{1}{p}$.
\end{enumerate}
The \textbf{best} constant $C$ in  \eqref{eqn.weighted} satisfies
\[C\leq B_1\begin{cases} (q')^{1/p'}q^{1/q}&\text{if } 1<p\leq q, q\geq 2\\
p^{1/q}(p^\prime)^{1/p'}&\text{if } 1<p\leq q<2\end{cases}\]
and $C\leq B_2q^{1/q}(p^\prime)^{1/q'}$ if $1<q<p<\infty$.
\end{theorem}
Although one cannot expect to obtain a sharp inequality by means of interpolation, it is still important to be able to control the optimal constant $C$ by means of a quantity $B_1$ determined by the weights $u$ and $v$. 

\begin{corollary}[Pitt's inequality for $\mathcal{F}_{k,a}$]
Assume $1<p\leq q<\infty$ and that the exponents $\alpha<0$ and $l>0$ satisfy the conditions
$\alpha>-\frac{2\left<k\right>+N+a-2}{q}$, $l<\frac{2\left<k\right>+N+a-2}{p}$ and
\begin{equation}\label{eq.homo-pitt}
\frac{1}{2\left<k\right>+N+a-2}(\alpha+l)=\frac{1}{p'}-\frac{1}{q}.
\end{equation}
Then the inequality
\[\Bigl(\int_{\R^N}\vert\mathcal{F}_{k,a}f(\xi)\vert^q\|\xi\|^{\alpha q}\,d\mu_{k,a}(\xi)\Bigr)^{1/q}\leq C\Bigl(\int_{\R^N}\vert f(x)\vert^p\| x\|^{lp}\,d\mu_{k,a}(x)\Bigr)^{1/p}\]
holds for all $f\in L^p_v(\R^N,\vartheta_{k,a}(x)dx)$.
\end{corollary}
\begin{proof}
This follows from Theorem \ref{thm.weighted} by choosing the weights $u(\xi)=\|\xi\|^{\alpha q}$, $v(x)=\|x\|^{lp}$, but where the rearrangements $u^*$ and $v^*$ are now taken with respect to the weighted measure $d\mu_{k,a}(x)=\vartheta_{k,a}(x)dx$ on $\R^N$. For $u$ we compute that
\[D_u(s)=\mu_{k,a}(\{\xi\in\R^N\,:\,\|\xi\|^{\alpha q}>s\}) = \mu_{k,a}(B_{s^{1/(\alpha q)}}(0)) = c(s^{1/(\alpha q)})^{a(\nu_a+1)} = cs^{(2\left<k\right>+N+a-2)/(\alpha q)},\]
from which it follows that $u^*(t)=c_\alpha t^{\alpha q/(2\left<k\right>+N+a-2)}$. Analogously, $(1/v)^*(t)=c_l t^{-lp/(2\left<k\right>+N+a-2)}$, so that
\[\begin{split}
\int_0^{1/s}u^*(t)\,dt&=\frac{c_\alpha}{1+\frac{\alpha q}{2\left<k\right>+N+a-2}}=s^{-\bigl(1+\frac{\alpha q}{2\left<k\right>+N+a-2}\bigr)},\\
\int_0^s ((1/v)^*(t))^{p^\prime-1}\,dt&= \frac{c_l^{p^\prime-1}}{1-\frac{lp(p^\prime-1)}{2\left<k\right>+N+a-2}} s^{1-\frac{lp(p^\prime-1)}{2\left<k\right>+N+a-2}}.
\end{split}\]
Since these quantities are required to be finite, in particular, we arrive at the first two conditions stated in the Corollary. The third condition comes from the simple observation that the combines exponent in
\[\Bigl(s^{-\frac{2\left<k\right>+N+a-2+\alpha q}{2\left<k\right>+N+a-2}}\Bigr)^{1/q} \Bigl(s^{\frac{2\left<k\right>+N+a-2-lp(p^\prime-1)}{2\left<k\right>+N+a-2}}\Bigr)^{1/p^\prime} \] must be zero, resulting in the condition
\[\frac{1}{2\left<k\right>+N+a-2}\Bigl(\alpha+l\frac{p(p^\prime-1)}{p^\prime}\Bigr)= -\frac{1}{q}+\frac{1}{p^\prime}.\]
\end{proof}
\noindent  The choice of weights results in a particularly simple 'homogeneity condition' \eqref{eq.homo-pitt} but could also have been carried out for the weights $u(\xi)=\|\xi\|^\alpha$, $v(x)=\|x\|^l$. We leave it to the interested reader to write out the conditions that $\alpha$ and $l$ must satisfy in this case.
\medskip

Beckner's logarithmic inequality followed from an endpoint differentiation argument applied to the sharp Pitt's inequality in the special case $p=q=2$. Since the interpolation techniques used above do not produce optimal constants, we cannot obtain the logarithmic inequality either. This was recently done by different techniques in \cite{GIT2}, which therefore settles a question that we raised in a previous version of the present paper.

\begin{remark}
The scope of the aforementioned paper \cite{Benedetto-Heinig} by Benedetto and Heinig is considerably wider than what we have suggested above. Indeed, the nature of the weight conditions enforced is such that one can work with the $A_p$-weights of Muckenhoupt. Since the measure space $(\R^N,\vartheta_{k,a})$ is doubling, there is a vast machinery available to produce further $A_p$-weighted inequalities for $\mathcal{F}_{k,a}$. We have decided against such applications, since they would seem somewhat tangential to our main applications: classical weighted inequalities and applications to uncertainty principles.  
\end{remark}

\section{Qualitative nonconcentration uncertainty principles}
The previous sections have presented several versions of the Heisenberg--Pauli--Weyl uncertainty principle and a strengthening in terms of entropy. The present section collects uncertainty principles that follow directly from \cite{Ghobber-Jaming-studia}. The purpose will not be to repeat their arguments but merely to point out the fact that one obtains uncertainty principle in addition to those already established. In all of the following results, $N$, $k$, and $a$ are required to satisfy the conditions in lemma \ref{lemma.B-parameters}.

\begin{theorem}[Benedicks--Amrein--Berther principle] Let $S,V$ be measurable subsets of $\R^N$ with $\mu_{k,a}(S),\mu_{k,a}(V)<\infty$. There exists a constant $C=C(k,a,S,V)$ such that for all $f\in L^2_{k,a}(\R^N)$
\[\|f\|^2_{L^2_{k,a}} \leq C\Bigl(\|f\|^2_{L^2_{k,a}(\R^N\setminus S)} + \|\mathcal{F}_{k,a}f\|^2_{L^2_{k,a}(\R^N\setminus V)}\Bigr).\]
\end{theorem}

We include the following analogue of the Matolcsi--Sz\"ucs inequality for completeness, although it is morally much weaker than the Benedicks--Amrein-Berthier result. The latter result can be obtained directly from an adaptation of the methods in \cite{Ghobber-Jaming-studia}.

\begin{prop}
If $f\in L^2(\R^N,\vartheta_{k,a})$ is nonzero, then $\mu_{k,a}(A_f)\cdot\mu_{k,a}(A_{\mathcal{F}_{k,a}f})\geq 1$, where $A_f=\{x\in\R^N\,:\,f(x)\neq 0\}$ and $A_{\mathcal{F}_{k,a}f}=\{\xi\in\R^N\,:\,\mathcal{F}_{k,a}f(\xi)\neq 0\}$.
\end{prop}
\begin{proof}
For an arbitrary $\mu_{k,a}$-measurable subset $E\subset\R^N$ it follows from the inequality $\|\mathcal{F}_{k,a}f\|_{\infty}\leq\|f\|_{L^1_{k,a}}$ that
\[\begin{split}
\int_E|\mathcal{F}_{k,a}f(\xi)|^2\,d\mu_{k,a}(\xi) &\leq \mu_{k,a}(E)\|\mathcal{F}_{k,a}f\|^2_{\infty} \leq \mu_{k,a}(E)\|f\|^2_{L^1_{k,a}}\\
&\leq\mu_{k,a}(E)\Bigl(\int_{\R^N}\mathbf{1}_{A_f}(x)\,d\mu_{k,a}(x)\Bigr)^{\frac{1}{2}\cdot 2} \Bigl(\int_{\R^N}|f(x)|^2\,d\mu_{k,a}(x)\Bigr)^{\frac{1}{2}\cdot 2}\\
&=\mu_{k,a}(E)\mu_{k,a}(A_f)\|f\|^2_{L^2_{k,a}}.
\end{split}\]
In particular, with $E=A_{\mathcal{F}_{k,a}f}$, we conclude that
\[\mu_{k,a}(A_{\mathcal{F}_{k,a}f})\cdot\mu_{k,a}(A_f)\|f\|^2_{L^2_{k,a}}\geq\int_{A_{\mathcal{F}_{k,a}f}}|\mathcal{F}_{k,a}f(\xi)|^2\,d\mu_{k,a}(\xi),\]
that is, $\mu_{k,a}(A_{\mathcal{F}_{k,a}f})\cdot\mu_{k,a}(A_f)\geq 1$, by the Plancherel theorem for $\mathcal{F}_{k,a}$.
\end{proof}

\begin{remark}
A stronger formulation of the nonconcentration property of $\mathcal{F}_{k,a}$ is captured by the Logvinenko--Sereda theorem (cf. \cite[Section~10.3]{Muscalu-Schlag-I}), which was recently obtained for the Hankel transform in \cite{Ghobber-Jaming-LS}. By previous remarks, this extends to a result for $\mathcal{F}_{k,a}$ acting on radial functions in $L^2_{k,a}(\R^N)$. We intend to address the more general case of arbitrary $L^2$-functions in the near future.
\end{remark}

Further uncertainty principles include 
\begin{itemize}
\item[(a)] a local uncertainty principle, which implies the Heisenberg--Pauli--Weyl uncertainty principle; 
\item[(b)] qualitative uncertainty principles analogous to the Benedicks--Amrein--Berthier principle and the Donoho--Stark principle
\end{itemize}

In (a), one obtains the following version of the uncertainty principle which generalizes Theorem 5.29 in \cite{BSKO} to include different powers of the norms involved.

\begin{corollary}[Global uncertainty principle] For $s,\beta>0$ there exists a constant $c_{s,\beta,k,a}$ such that for all $f\in L^2(\R^N,\vartheta_{k,a}(x)dx)$
\[\bigl\| \vert\cdot\vert^sf\bigr\|^{\frac{2\beta}{s+\beta}}_{L^2(\R^N,\vartheta_{k,a})} \cdot \bigl\| \vert\cdot\vert\mathcal{F}_{k,a}f\bigr\|^{\frac{2s}{s+\beta}}_{L^2(\R^N,\vartheta_{k,a})} \geq c_{s,\beta,k,a} \|f\|^2_{L^2(\R^N,\vartheta_{k,a})}.\]
\end{corollary}

\begin{remark}
The Dunkl-case $a=2$ was recently obtained by Soltani \cite{Soltani-HPW} by a different method.
\end{remark}

Having already mentioned the analogue of the Benedicks--Amrein-Berthier result, we conclude by returning to our starting point, the Donoho--Stark uncertainty principle. 
\begin{definition}
Let $S$ and $\Sigma$ be measurable subsets of $\R^N$ with $\mu_{k,a}(S),\mu_{k,a}(\Sigma)<\infty$, and let $\varepsilon,\delta\geq 0$ be given. A function $f\in L^p(\R^N,\vartheta_{k,a})$ is \emph{$(L^p,\varepsilon)$-concentrated}  on $S$ if $\|f-\mathbf{1}_Sf\|_p\leq\varepsilon\|f\|_p$.  A function $f\in L^p(\R^N,\vartheta_{k,a})$ is \emph{$(L^p,\delta)$-bandlimited to $\Sigma$} if the exists a function $f_\Sigma\in L^p(\R^N,\vartheta_{k,a})$ with $\mathrm{supp}\,\mathcal{F}_{k,a}(f_\Sigma)\subset \Sigma$ such that $\|f-f_\Sigma\|_p\leq\delta\|f\|_p$.
\end{definition}

\begin{theorem}[Dohono--Stark principle] Let $S$ and $V$ be measurable subsets of $\R^N$, and let $f\in L^2_{k,a}(\R^N)$ be of unit $L^2$-norm, $\varepsilon$-concentrated on $S$ and $\delta$-bandlimited on $V$ for the $(k,a)$-generalized Fourier transform $\mathcal{F}_{k,a}$. Then
\[\mu_{k,a}(S)\mu_{k,a}(V)\geq\frac{\bigl(1-\sqrt{\varepsilon^2+\delta^2}\bigr)^2}{c_{k,a}^2}.\]
\end{theorem}
In the Dunkl-case $a=2$ the result is due to Ghobber and Jaming, while a slightly less precise bound from below was obtained by Kawazoe and Mejjaoli in \cite{Kawazoe-Mejjaoli} (several variants appear in their Section 8, together with some historical remarks). We have recently extended these results to the Heckman--Opdam transform associated to certain higher rank root systems in $\R^N$, cf. \cite{Johansen-unc}.

\section{Open problems}
The sharp Hausdorff--Young inequality for the Hankel transform would imply a sharp entropic inequality for $\mathcal{F}_{k,a}$ acting on radial $L^2$-functions in $\R^N$. We are not aware of a reliable source, however, so this remains an interesting open problem. More generally, one would like to have a sharp Hausdorff--Young inequality for $\mathcal{F}_{k,a}$ acting on arbitrary $L^2$-functions but this appears to be out of reach at the moment. In the case $a=2$, however, it seems likely that such a result can be obtained in the special case where the Weyl group $W$ associated with the underlying root system is isomorphic to $\Z_2^N$, since a tensorization technique already used by Beckner would reduce to problem to the one-dimensional case, which seems doable.

As already mentioned, the sharp Pitt's inequality for $\mathcal{F}_{k,a}$ has recently been established in \cite{GIT2}. The authors do not obtain a sharp logarithmic uncertainty inequality, however. 

We finally wish to point out that the general framework of \cite{BSKO} has been extended to include Clifford algebra-valued functions on $\R^N$ (cf. \cite{deBie-et.al.} and \cite{deBie-trans}) and to more general integral transforms in \cite{deBie15}. The liberal use of interpolation techniques in the present paper were scalar-valued in nature but there are many extensions of classical interpolation theory to operator- or vector-valued functions. It is therefore to be expected that many of the results we have obtained should have immediate extensions to the Clifford-algebra-valued setting. The methods are applicable in the framework of \cite{deBie15} (which the authors also acknowledge). It would be interesting to search for a sharp Hausdorff--Young inequality in this setup.
\appendix
\section{Proof of theorem \ref{thm.weighted}}
The present appendix establishes theorem \ref{thm.weighted}. The proof is largely contained in \cite[Section~2]{Benedetto-Heinig} where the details were written out in the case of the Euclidean Fourier transform on $\R^N$. As the authors remark at the beginning of section 2, loc. cit., and expounded upon in their remark 6c and d, the result (that is, the weighted inequality in theorem \ref{thm.weighted}) is essentially valid for any bounded linear operator of type $(1,\infty)$ and $(2,2)$. Benedetto and Heinig clearly had in mind an Euclidean setup, where Lebesgue measure was used, but some of the references they list -- most notably \cite{Jodeit-Torchinsky} -- indeed involve $L^p$-spaces with respect to weighted Lebesgue measure. Of course $d\mu_{k,a}$ is also a weighted Lebesgue measure, but we found it impractical to incorporate the density $\vartheta_{k,a}$ in the weights $u$ and $v$. Although $\mathcal{F}_{k,a}$ is of type $(2,2)$ also from $L^2_{v_{k,a}}(\R^N,dx)$ to $L^2_{u_{k,a}}(\R^N,d\xi)$, where $u_{k,a}(\xi)=\vartheta_{k,a}(\xi)$ and $v_{k,a}(x)=\vartheta_{k,a}(x)$, it seems difficult to compute the decreasing rearrangements of the power weights $u(\xi)=\|\xi\|^\alpha\vartheta_{k,a}(\xi)$ and $v(x)=\|x\|^l\vartheta_{k,a}(x)$ that would be used in the original formulation of \cite[Theorem~1]{Benedetto-Heinig}. With this approach it is clear that a Pitt-type inequality for $\mathcal{F}_{k,a}$ should hold, but it is difficult to determine the exact range of exponents $\alpha,l$ and powers $p,q$ for which the inequality is valid.

\begin{lemma}[Hardy's lemma]
Let $\psi$ and $\chi$ be non-negative Lebesgue measurable functions on $(0,\infty)$, and assume 
\[\int_0^s\psi(t)\,dt\leq\int_0^s\chi(t)\,dt\]
for all $s>0$. If $\varphi$ is non-negative and non-decreasing on $(0,\infty)$, then
\[\int_0^\infty\varphi(t)\psi(t)\,dt\leq \int_0^\infty\varphi(t)\chi(t)\,dt.\]
\end{lemma}

Let us agree to let a weight  on a measure space $(X,\mu)$ is  a non-negative $\mu$-locally integrable functions on $X$. 

\begin{theorem*}[Theorem A in \cite{Benedetto-Heinig}]
Let $u$ and $v$ be weight functions on $(0,\infty)$ and suppose $1<p,q<\infty$. There exists a positive constant $C$ such that for all non-negative Lebesgue measurable functions $f$ on $(0,\infty)$ the weighted Hardy inequality
\begin{equation}\label{eqn.weighted-Hardy}
\Bigl(\int_0^\infty\Bigl(\int_0^t f\Bigr)^qu(t)\,dt\Bigr)^{1/q}\leq C\Bigl(\int_0^\infty f(t)^pv(t)\,dt\Bigr)^{1/p}\end{equation}
is satisfied if and only if
\begin{enumerate}[label=(\roman*)]
\item for $1<p\leq q<\infty$,
\[\sup_{s>0}\Bigl(\int_s^\infty u(t)\,dt\Bigr)^{1/q}\Bigl(\int_0^s v(t)^{1-p'}dt\Bigr)^{1/p'}\equiv A_1<\infty,\]
and
\item for $1<q<p<\infty$,
\[\Bigl(\int_0^\infty\Bigl(\int_s^\infty u\Bigr)^{r/q}\Bigl(\int_0^sv^{1-p'}\Bigr)^{r/q'}v(s)^{1-p'}ds\Bigr)^{1/r}\equiv A_2<\infty\]
where $\frac{1}{r}=\frac{1}{q}-\frac{1}{p}$.
\end{enumerate} Moreover, if $C$ is the best constant in the weighted Hardy inequality, then in case (i) we have $A_1\leq C\leq A_1(q')^{1/p'}q^{1/q}$, and in case (ii) we have $(\frac{p-q}{p-1})^{1/q'}q^{1/q}A_2\leq C\leq (p')^{1/q'}q^{1/q}A_2$.
\end{theorem*}

The proof of \cite[Theorem 1]{Benedetto-Heinig} relies of several classical rearrangement inequalities. Since Benedetto and Heinig formulate these for Lebesgue measure and the Fourier transform, two of their results must be modified slightly. The \emph{decreasing rearrangement} of $f$ defined on $(X,\mu)$ is the function $f^*:[0,\infty)\to[0,\infty)$ defined by $f^*(t)=\inf\{s\geq 0\,:\, D_f(s)\leq t\}$. By convention $\inf\emptyset=\infty$, so that $f^*(t)=\infty$ whenever $D_f(s)>t$ for all $s\in[0,\infty)$.

\begin{lemma}[The Hardy--Littlewood rearrangement inequality]
Let $f$ and $g$ be non-negative $\mu_{k,a}$-measurable functions on $\R^N$. Then
\[\int_{\R^N}f(x)g(x)\,d\mu_{k,a}(x)\leq\int_0^\infty f^*(t)g^*(t)\,dt\]
and
\[\int_0^\infty f^*(t)\frac{1}{(1/g)^*(t)} dt\leq \int_{\R^N}f(x)g(x)\,d\mu_{k,a}(x).\]
\end{lemma}
\begin{proof}
The first statement can be found as Theorem~2.2 on page 44 in \cite{Bennett-Sharpley}.
\end{proof}

\begin{lemma}[Theorem~B in \cite{Benedetto-Heinig}; the type estimate of Jodeit and Torchinsky]
Let $q\geq 2$. There is a constant $K_q>0$ such that, for all $f\in(L^1+L^2)(\R^N,\vartheta_{k,a}(x)dx)$ and for all $s\geq 0$, the inequality
\[\int_0^s[(\mathcal{F}_{k,a})^*(t)]^q\,dt \leq K_q^q\int_0^s\Bigl(\int_0^{1/t}f^*(r)\,dr\Bigr)^q\,dt\]
holds.
\end{lemma}
\begin{proof}
The case $q=2$ is \cite[Theorem~4.6]{Jodeit-Torchinsky} and the more general statement for $q\geq 2$ is \cite[Theorem~4.7]{Jodeit-Torchinsky}.
\end{proof} 
Jodeit and Torchinsky phrased their results more generally in terms of sublinear operators $T$ acting between Orlicz spaces $L_A(\R^n,d\mu)$ and $L_B(\R^N,d\nu)$, where $A$ and $B$ are Young functions. The aforementioned result is obtained by considering power weights as Young functions and using that we already know that $\mathcal{F}_{k,a}$ is of type $(1,\infty)$ and $(2,2)$ when using the weighted measure $\mu=\nu=\mu_{k,a}$. A close inspection of \cite[Section~2]{Jodeit-Torchinsky} establishes that they form the symmetric rearrangements $f^*$ and $(Tf)^*$ with respect to $\mu$ and $\nu$, respectively, so we do not have to redo  their proofs.

\begin{theorem}
Let $u$ and $v$ be weight functions on $\R^n$, suppose $1<p,q<\infty$, and let $K$ be the constant from theorem \ref{thm.B} associated with the relevant index $\geq 2$. There is a positive constant $C$ such that, for all $f\in L^p_v(\R^N,\vartheta_{k,a}(x)dx)$, the inequality 
\begin{equation}\label{eqn.weighted}
\Bigl(\int_{\R^N}\vert\mathcal{F}_{k,a}f(\xi)\vert^q u(\xi)\,d\mu_{k,a}(\xi)\Bigr)^{1/q}\leq KC\Bigl(\int_{\R^N}\vert f(x)\vert^pv(x)\,d\mu_{k,a}(x)\Bigr)^{1/p}
\end{equation}
holds in the following ranges and with the following constraints on $u$ and $v$:
\begin{enumerate}[label=(\roman*)]
\item $1<p\leq q<\infty$ and 
\[\sup_{s>0}\Bigl(\int_0^{1/s}u^*(t)\,dt\Bigr)^{1/q}\Bigl(\int_0^s\bigl(\bigl(\tfrac{1}{v}\bigr)^*\bigr)(t)^{p^\prime-1}\,dt\Bigr)^{1/p'}\equiv B_1<\infty;\]
\item for $1<q<p<\infty$ and 
\[\Bigl(\int_0^\infty\Bigl(\int_0^{1/s}u^*\Bigr)^{r/q}\Bigl(\int_0^s\bigl(\tfrac{1}{v}\bigr)^{*(p'-1)}\Bigr)^{r/q'}\bigl(\bigl(\tfrac{1}{v}\bigr)^*\bigr)(s)^{p'-1}ds\Bigr)^{2/r}\equiv B_2<\infty,\]
where $\frac{1}{r}=\frac{1}{q}-\frac{1}{p}$.
\end{enumerate}
The \textbf{best} constant $C$ in  \eqref{eqn.weighted} satisfies
\[C\leq B_1\begin{cases} (q')^{1/p'}q^{1/q}&\text{if } 1<p\leq q, q\geq 2\\
p^{1/q}(p^\prime)^{1/p'}&\text{if } 1<p\leq q<2\end{cases}\]
and $C\leq B_2q^{1/q}(p^\prime)^{1/q'}$ if $1<q<p<\infty$.
\end{theorem}

The proof follows exactly as in \cite[section~2]{Benedetto-Heinig}, except that we replace their Theorem B with the above version for $d\mu_{k,a}$, and use the Hardy--Littlewood rearrangement inequality for rearrangements with respect to the weighted measure $d\mu_{k,a}$ rather than Lebesgue measure on $\R^N$.

\begin{small}
\def\cprime{$'$}
\providecommand{\bysame}{\leavevmode\hbox to3em{\hrulefill}\thinspace}
\providecommand{\MR}{\relax\ifhmode\unskip\space\fi MR }
\providecommand{\MRhref}[2]{%
  \href{http://www.ams.org/mathscinet-getitem?mr=#1}{#2}
}
\providecommand{\href}[2]{#2}

\end{small}

\begin{thebibliography}{DB{\O}SS13}

\bibitem[AASS09]{Anker-besov}
C.~Abdelkefi, J.-P. Anker, F.~Sassi, and M.~Sifi, \emph{Besov-type spaces on
  {$\Bbb R^d$} and integrability for the {D}unkl transform}, SIGMA Symmetry
  Integrability Geom. Methods Appl. \textbf{5} (2009), Paper 019, 15.

\bibitem[AB77]{Amrein-Berthier}
W.~O. Amrein and A.~M. Berthier, \emph{On support properties of
  {$L^{p}$}-functions and their {F}ourier transforms}, J. Functional Analysis
  \textbf{24} (1977), no.~3, 258--267.

\bibitem[Bab61]{Babenko}
K.~I. Babenko, \emph{An inequality in the theory of {F}ourier integrals}, Izv.
  Akad. Nauk SSSR Ser. Mat. \textbf{25} (1961), 531--542.

\bibitem[BB06]{Physics}
I.~Bialynicki-Birula, \emph{Formulation of the uncertainty relations in terms
  of the {R}\'enyi entropies}, Phys. Rev. A (3) \textbf{74} (2006), no.~5,
  052101, 6.

\bibitem[BBM75]{BBM}
I.~Bia{\l}ynicki-Birula and J.~Mycielski, \emph{Uncertainty relations for
  information entropy in wave mechanics}, Comm. Math. Phys. \textbf{44} (1975),
  no.~2, 129--132.

\bibitem[BDSS05]{BrackX}
F.~Brackx, N.~De~Schepper, and F.~Sommen, \emph{The {C}lifford-{F}ourier
  transform}, J. Fourier Anal. Appl. \textbf{11} (2005), no.~6, 669--681.

\bibitem[Bec75]{Beckner-inequalities}
W.~Beckner, \emph{Inequalities in {F}ourier analysis}, Ann. of Math. (2)
  \textbf{102} (1975), no.~1, 159--182.

\bibitem[Bec95]{Beckner-Pitt1}
\bysame, \emph{Pitt's inequality and the uncertainty principle}, Proc. Amer.
  Math. Soc. \textbf{123} (1995), no.~6, 1897--1905.

\bibitem[Ben85]{Benedicks}
M.~Benedicks, \emph{On {F}ourier transforms of functions supported on sets of
  finite {L}ebesgue measure}, J. Math. Anal. Appl. \textbf{106} (1985), no.~1,
  180--183.

\bibitem[BH03]{Benedetto-Heinig}
J.~J. Benedetto and H.~P. Heinig, \emph{Weighted {F}ourier inequalities: new
  proofs and generalizations}, J. Fourier Anal. Appl. \textbf{9} (2003), no.~1,
  1--37.

\bibitem[BS88]{Bennett-Sharpley}
C.~Bennett and R.~Sharpley, \emph{Interpolation of operators}, Pure and Applied
  Mathematics, vol. 129, Academic Press, Inc., Boston, MA, 1988.

\bibitem[BS15]{BS-semi}
S.~Ben~Sa{\"{\i}}d, \emph{Strichart estimates for the Schr\"odinger--Laguerre operators}, Semigroup Forum \textbf{90} (2015), 251--269. 

\bibitem[BSK{\O}09]{BSKO-cr}
S.~Ben~Sa{\"{\i}}d, T.~Kobayashi, and B.~{\O}rsted, \emph{Generalized {F}ourier transforms
{$\mathcal{F}_{k,a}$}}, C. R. Acad. Sci. Paris, Ser I \textbf{347} (2009), 1119--1124.

\bibitem[BSK{\O}12]{BSKO}
\bysame, \emph{Laguerre semigroup
  and {D}unkl operators}, Compos. Math. \textbf{148} (2012), no.~4, 1265--1336.

\bibitem[Cha00]{Chatterji}
S.~D. Chatterji, \emph{Remarks on the {H}ausdorff-{Y}oung inequality}, Enseign.
  Math. (2) \textbf{46} (2000), no.~3-4, 339--348.

\bibitem[CP84]{Cowling-Price}
M.~G. Cowling and J.~F. Price, \emph{Bandwidth versus time concentration: the
  {H}eisenberg-{P}auli-{W}eyl inequality}, SIAM J. Math. Anal. \textbf{15}
  (1984), no.~1, 151--165.

\bibitem[CRS07]{Ciatti-Ricci-Sundari}
P.~Ciatti, F.~Ricci, and M.~Sundari, \emph{Heisenberg-{P}auli-{W}eyl
  uncertainty inequalities and polynomial volume growth}, Adv. Math.
  \textbf{215} (2007), no.~2, 616--625.

\bibitem[DB12]{deBie-overview}
H.~De~Bie, \emph{Clifford algebras, {F}ourier transforms, and quantum
  mechanics}, Math. Methods Appl. Sci. \textbf{35} (2012), no.~18, 2198--2228.

\bibitem[DB13]{deBie}
\bysame, \emph{The kernel of the radially transformed {F}ourier transform},
  Integral Transform. Spec. Funct. \textbf{24} (2013), no.~12, 1000--1008.

\bibitem[DBOvdJ15]{deBie15}
H.~De~Bie, R. Oste, and J. van der Jeugt, \emph{Generalized {F}ourier transforms arising from the enveloping algebras of $\mathfrak{sl}(2)$ and $\mathfrak{osp}(1\vert 2)$}, 2015, to appear in Int. Math. Res. Not. IMRN.

\bibitem[DB{\O}SS12]{deBie-trans}
H.~De~Bie, B.~{\O}rsted, P.~Somberg, and V.~Sou{\v{c}}ek, \emph{Dunkl operators
  and a family of realizations of {$\mathfrak{osp}(1\vert 2)$}}, Trans. Amer.
  Math. Soc. \textbf{364} (2012), no.~7, 3875--3902.

\bibitem[DB{\O}SS13]{deBie-et.al.}
\bysame, \emph{The {C}lifford deformation of the {H}ermite semigroup}, SIGMA
  Symmetry Integrability Geom. Methods Appl. \textbf{9} (2013), Paper 010, 22.

\bibitem[DBX11]{DeBie-Xu}
H.~De~Bie and Y.~Xu, \emph{On the {C}lifford-{F}ourier transform}, Int. Math.
  Res. Not. IMRN (2011), no.~22, 5123--5163.

\bibitem[dJ93]{deJeu-dunkl}
M.~F.~E. de~Jeu, \emph{The {D}unkl transform}, Invent. Math. \textbf{113}
  (1993), no.~1, 147--162.

\bibitem[dJ94]{deJeu-uncertainty}
\bysame, \emph{An uncertainty principle for integral operators}, J. Funct.
  Anal. \textbf{122} (1994), no.~1, 247--253.

\bibitem[Dha07]{Dhaouadi}
L.~ Dhaouadi: \emph{Heisenberg Uncertainty Principle for the $q$-Bessel Fourier transform}, preprint, 2007.

\bibitem[Dun91]{Dunkl91}
C.~F. Dunkl, \emph{Integral kernels with reflection group invariance}, Canad.
  J. Math. \textbf{43} (1991), no.~6, 1213--1227.

\bibitem[EK87]{Eguchi}
M. Eguchi and K. Kumahara, \emph{A Hardy--Littlewood theorem for spherical Fourier transforms on symmetric spaces}, J. Funct. Anal. \textbf{71} (1987), no. 1, 104--122.

\bibitem[FS97]{Folland-Sitaram}
G.~B. Folland and A.~Sitaram, \emph{The uncertainty principle: a mathematical
  survey}, J. Fourier Anal. Appl. \textbf{3} (1997), no.~3, 207--238.

\bibitem[Gho13]{Ghobber-L1}
S.~Ghobber, \emph{Uncertainty principles involving {$L^1$}-norms for the
  {D}unkl transform}, Integral Transforms Spec. Funct. \textbf{24} (2013),
  491--501.

\bibitem[GJ13]{Ghobber-Jaming-LS}
S.~Ghobber and P.~Jaming, \emph{The {L}ogvinenko-{S}ereda theorem for the
  {F}ourier-{B}essel transform}, Integral Transforms Spec. Funct. \textbf{24}
  (2013), no.~6, 470--484.

\bibitem[GJ14]{Ghobber-Jaming-studia}
\bysame, \emph{Uncertainty principles for integral operators}, Studia Math.
  \textbf{220} (2014), no.~3, 197--220.

\bibitem[GIT15a]{GIT1}
D. Gorbachev, V. Ivanov, and S. Tikhonov: \emph{Sharp Pitt inequality and logarithmic uncertainty principle for Dunkl transform in $L^2$}, arXiv: 1505.02958, appears in Journal of Approximation Theory.

\bibitem[GIT15b]{GIT2}
\bysame, \emph{Pitt's inequalities and uncertainty principle for generalized Fourier transform}, arXiv: 1507.06445.

\bibitem[HL27]{Hardy-Littlewood}
G. H. Hardy and J. E. Littlewood, \emph{Some new properties of Fourier constants}, Math. Ann. textbf{97} (1927), no. 1, 159--209.

\bibitem[Hir57]{Hirschman-entropy}
I.~I. Hirschman, Jr., \emph{A note on entropy}, Amer. J. Math. \textbf{79}
  (1957), 152--156.

\bibitem[How88]{Howe88}
\bysame, \emph{The oscillator semigroup}, The mathematical heritage of
  {H}ermann {W}eyl ({D}urham, {NC}, 1987), Proc. Sympos. Pure Math., vol.~48,
  Amer. Math. Soc., Providence, RI, 1988, pp.~61--132.

\bibitem[Joh15a]{Johansen-HY}
T.~R. Johansen, \emph{{H}ardy--{L}ittlewood inequalities for the
  {H}eckman--{O}pdam transform}, \url{http://arxiv.org/abs/1501.06513} (2015).

\bibitem[Joh15b]{Johansen-CO}
\bysame, \emph{Remarks on the inverse {C}herednik--{O}dam transform on the real
  line}, \url{http://arxiv.org/abs/1502.01293} (2015).

\bibitem[Joh15c]{Johansen-unc}
\bysame, \emph{Uncertainty principles for the {H}eckman--{O}pdam transformm},
  preprint, submitted (2015).

\bibitem[JT71]{Jodeit-Torchinsky}
M.~Jodeit, Jr. and A.~Torchinsky, \emph{Inequalities for {F}ourier transforms},
  Studia Math. \textbf{37} (1970/71), 245--276.

\bibitem[KM05]{Kobayashi-Mano-appl}
T.~Kobayashi and G.~Mano, \emph{Integral formulas for the minimal
  representation of {$\mathrm{O}(p,2)$}}, Acta Appl. Math. \textbf{86} (2005),
  no.~1-2, 103--113.

\bibitem[KM07]{Kobayashi-Mano-proc}
\bysame, \emph{Integral formula of the unitary inversion operator for the
  minimal representation of {${\rm O}(p,q)$}}, Proc. Japan Acad. Ser. A Math.
  Sci. \textbf{83} (2007), no.~3, 27--31.

\bibitem[KM10]{Kawazoe-Mejjaoli}
T.~Kawazoe and H.~Mejjaoli, \emph{Uncertainty principles for the {D}unkl
  transform}, Hiroshima Math. J. \textbf{40} (2010), no.~2, 241--268.

\bibitem[KM11]{Kobayashi-Mano}
T.~Kobayashi and G.~Mano, \emph{The {S}chr\"odinger model for the minimal
  representation of the indefinite orthogonal group {${\rm O}(p,q)$}}, Mem.
  Amer. Math. Soc. \textbf{213} (2011), no.~1000, vi+132.

\bibitem[LM99]{Laeng-Morpurgo}
E.~Laeng and C.~Morpurgo, \emph{An uncertainty inequality involving
  {$L^1$}-norms}, Proc. Amer. Math. Soc. \textbf{127} (1999), no.~12,
  3565--3572.

\bibitem[Mar10]{Martini}
A.~Martini, \emph{Generalized uncertainty inequalities}, Math. Z. \textbf{265}
  (2010), no.~4, 831--848.

\bibitem[Mor01]{Morpurgo}
C.~Morpurgo, \emph{Extremals of some uncertainty inequalities}, Bull. London
  Math. Soc. \textbf{33} (2001), no.~1, 52--58.

\bibitem[MS73]{Matolcsi-Szucs}
T.~Matolcsi and J.~Sz{\H{u}}cs, \emph{Intersection des mesures spectrales
  conjugu\'ees}, C. R. Acad. Sci. Paris S\'er. A-B \textbf{277} (1973),
  A841--A843.

\bibitem[MS13]{Muscalu-Schlag-I}
C.~Muscalu and W.~Schlag, \emph{Classical and multilinear harmonic analysis.
  {V}ol. {I}}, Cambridge Studies in Advanced Mathematics, vol. 137, Cambridge
  University Press, Cambridge, 2013.

\bibitem[O'N63]{Oneil-convolution}
R.~O'Neil, \emph{Convolution operators and {$L(p,\,q)$} spaces}, Duke Math. J.
  \textbf{30} (1963), 129--142.

\bibitem[{\"O}P04]{OP}
M.~{\"O}zaydin and T.~Przebinda, \emph{An entropy-based uncertainty principle
  for a locally compact abelian group}, J. Funct. Anal. \textbf{215} (2004),
  no.~1, 241--252.

\bibitem[R{\"o}s99]{Rosler-duke}
M.~R{\"o}sler, \emph{Positivity of {D}unkl's intertwining operator}, Duke Math.
  J. \textbf{98} (1999), no.~3, 445--463.

\bibitem[RV98]{Rosler-Voit-Markov}
M.~R{\"o}sler and M.~Voit, \emph{Markov processes related with {D}unkl
  operators}, Adv. in Appl. Math. \textbf{21} (1998), no.~4, 575--643.

\bibitem[Sol13]{Soltani-HPW}
F.~Soltani, \emph{A general form of {H}eisenberg-{P}auli-{W}eyl uncertainty
  inequality for the {D}unkl transform}, Integral Transforms Spec. Funct.
  \textbf{24} (2013), no.~5, 401--409.

\bibitem[SW71]{Stein-Weiss-analysis}
E.~M. Stein and G.~Weiss, \emph{Introduction to {F}ourier analysis on
  {E}uclidean spaces}, Princeton University Press, Princeton, N.J., 1971,
  Princeton Mathematical Series, No. 32.

\bibitem[Tit48]{Titchmarsh}
E.~C. Titchmarsh, \emph{Introduction to the theory of {F}ourier integrals}, 2nd
  ed., Oxford University Press, 1948.

\bibitem[Wol92]{Wolf-Nova}
J.~A. Wolf, \emph{The uncertainty principle for {G}el\cprime fand pairs}, Nova
  J. Algebra Geom. \textbf{1} (1992), no.~4, 383--396.

\bibitem[Wol94]{Wolf-cayley}
\bysame, \emph{Uncertainty principles for {G}el\cprime fand pairs and {C}ayley
  complexes}, 75 years of {R}adon transform ({V}ienna, 1992), Conf. Proc.
  Lecture Notes Math. Phys., IV, Int. Press, Cambridge, MA, 1994, pp.~271--292.

\end{thebibliography}
\end{document}